\documentclass[a4paper,11pt,reqno]{amsart}
\usepackage{amssymb,mathrsfs}
\usepackage{tikz-cd}
\tikzcdset{arrow style=tikz,diagrams={>={Straight Barb[scale=0.8]}}}

\newtheorem{theorem}{Theorem}[section]
\newtheorem{proposition}[theorem]{Proposition}
\newtheorem{lemma}[theorem]{Lemma}
\newtheorem{corollary}[theorem]{Corollary}

\theoremstyle{definition}
\newtheorem{definition}[theorem]{Definition}

\newcommand{\bA}{\mathbf{A}}
\newcommand{\bC}{\mathbf{C}}
\newcommand{\bN}{\mathbf{N}}
\newcommand{\bQ}{\mathbf{Q}}
\newcommand{\bR}{\mathbf{R}}
\newcommand{\bZ}{\mathbf{Z}}
\newcommand{\cF}{\mathcal{F}}
\newcommand{\cS}{\mathcal{S}}
\newcommand{\fa}{\mathfrak{a}}
\newcommand{\fb}{\mathfrak{b}}
\newcommand{\fd}{\mathfrak{d}}
\newcommand{\fm}{\mathfrak{m}}
\newcommand{\fD}{\mathfrak{D}}
\newcommand{\sI}{\mathscr{I}}
\newcommand{\sJ}{\mathscr{J}}
\newcommand{\sO}{\mathscr{O}}
\newcommand{\sfd}{\mathsf{d}}
\newcommand{\sfI}{\mathsf{I}}
\newcommand{\sfJ}{\mathsf{J}}
\newcommand{\abs}[1]{{|{#1}|}}
\newcommand{\rd}[1]{{\lfloor{#1}\rfloor}}
\newcommand{\ru}[1]{{\lceil{#1}\rceil}}

\DeclareMathOperator{\gr}{gr}
\DeclareMathOperator{\mld}{mld}
\DeclareMathOperator{\ord}{ord}

\DeclareMathOperator{\sHom}{\mathscr{H}\!\mathit{om}}
\DeclareMathOperator{\Spec}{Spec}
\DeclareMathOperator{\wt}{wt}

\begin{document}
\title{Minimal log discrepancies on a fixed threefold}

\author{Masayuki Kawakita}
\address{Research Institute for Mathematical Sciences, Kyoto University, Kyoto 606-8502, Japan}
\email{masayuki@kurims.kyoto-u.ac.jp}
\thanks{Partially supported by JSPS Grant-in-Aid for Scientific Research (C) 24K06667.}

\begin{abstract}
We prove the ACC for minimal log discrepancies on an arbitrary fixed threefold.
\end{abstract}

\maketitle

\section{Introduction}
Shokurov \cite{Sh04} proved that the ACC for minimal log discrepancies together with the lower semi-continuity implies the termination of flips. In our preceding paper \cite{Kaxv}, we completed the ACC for minimal log discrepancies on a smooth threefold. The purpose of the present paper is to extend it to an arbitrary fixed threefold. ACC stands for the ascending chain condition whilst DCC stands for the descending chain condition. For a pair $(X,\Delta)$ and a scheme-theoretic point $\eta$ in $X$, $\mld_\eta(X,\Delta)$ denotes the minimal log discrepancy of $(X,\Delta)$ at $\eta$. For a subset $I$ of the positive real numbers, $I\in\Delta$ means that all coefficients in $\Delta$ belong to $I$.

\begin{theorem}\label{thm:main}
Fix a normal threefold $X$ and a subset $I$ of the positive real numbers which satisfies the DCC\@. Then the set
\[
\{\mld_\eta(X,\Delta)\mid\textup{$\eta$ scheme-theoretic point},\ \textup{$(X,\Delta)$ pair},\ \Delta\in I\}
\]
satisfies the ACC\@.
\end{theorem}

On a fixed germ, the ACC for minimal log discrepancies is equivalent to the ACC for $a$-lc thresholds, Musta\c{t}\u{a}'s uniform $\fm$-adic semi-continuity and Nakamura's boundedness \cite{K21}. Hence the extension applies to all of them as below.

\begin{theorem}
Fix a normal threefold $X$ and a subset $I$ of the positive real numbers which satisfies the DCC\@. Fix a non-negative real number $a$. Then the set
\[
\biggl\{t\in\bR\;\biggm|
\begin{array}{l}
\textup{$\eta$ scheme-theoretic point},\ \textup{$(X,\Delta)$ pair},\ \textup{$A$ $\bR$-Cartier},\\
\Delta\in I,\ A\in I,\ t\ge0,\ \mld_\eta(X,\Delta+tA)=a
\end{array}
\biggr\}
\]
satisfies the ACC\@.
\end{theorem}

\begin{theorem}\label{thm:adic}
Fix a normal threefold $X$ and a subset $I$ of the positive real numbers which satisfies the DCC\@. Then there exists a positive integer $l$ such that for a scheme-theoretic point $\eta$ in $X$, for an $\bR$-divisor $\Delta$ on $X$ with $\Delta\in I$ forming a pair $(X,\Delta)$ and for $\bR$-ideals $\fa=\prod_{j=1}^e\fa_j^{r_j}$ and $\fb=\prod_{j=1}^e\fb_j^{r_j}$ on $X$ with $r_j\in I$, if $\fa_j\sO_{X,\eta}+\fm_\eta^l=\fb_j\sO_{X,\eta}+\fm_\eta^l$ for all $j$, where $\fm_\eta$ denotes the maximal ideal in $\sO_{X,\eta}$, then $\mld_\eta(X,\Delta,\fa)=\mld_\eta(X,\Delta,\fb)$.
\end{theorem}

\begin{theorem}\label{thm:nakamura}
Fix a log terminal threefold $X$ and a subset $I$ of the positive real numbers which satisfies the DCC\@. Then there exists a positive integer $l$ such that for every scheme-theoretic point $\eta$ in $X$ and for every $\bR$-Cartier $\bR$-divisor $\Delta$ on $X$ with $\Delta\in I$, there exists a divisor $E$ over $X$ which computes $\mld_\eta(X,\Delta)$ and has the log discrepancy $a_E(X)\le l$.
\end{theorem}

The ACC for minimal log discrepancies is predicted without fixing a variety. In dimension three, it is known for minimal log discrepancies greater than $1-\varepsilon$, where $\varepsilon$ is a positive real number depending on the set of coefficients in the boundaries \cite{HLLaxv}. If the boundaries are zero, then one can take $\varepsilon$ to be $1/6$ \cite{LLaxv}. The ACC is also known in the case when the pairs admit a bounded number of exceptional divisors with log discrepancy at most one \cite{HLapp}.

We shall explain the proof of the theorems. After reduction to the case when $x\in X$ is the germ at a closed point of a $\bQ$-factorial terminal threefold, we aim at Nakamura's boundedness for a sequence of boundaries $\Delta_i$ stated as Theorem~\ref{thm:reduction}, that is, the existence of a bound $l$ such that for every $i$ there exists a divisor $E_i$ over $X$ which computes $\mld_x(X,\Delta_i)$ and has $a_{E_i}(X)\le l$. Extending our former work \cite{K15}, we may assume that $\mld_x(X,\Delta_i)$ is less than one. We seek a birational model $y_i\in Y_i\to x\in X$ such that it admits a pair $(Y_i,\Gamma_i)$ crepant to $(X,\Delta_i)$ and such that $\mld_x(X,\Delta_i)$ equals $\mld_{y_i}(Y_i,\Gamma_i)$ at a smooth point $y_i$ of $Y_i$. Once one obtains the model for infinitely many indices $i$ in a family, one can replace the pair $(X,\Delta_i)$ by $(Y_i,\Gamma_i)$ and apply the boundedness on a smooth threefold.

A uniform construction of models follows from the construction of a model of the pair $(\hat X,\hat\Delta)$ for the generic limit $\hat\Delta$ of the sequence of $\Delta_i$ defined on the spectrum $\hat x\in\hat X$ of a complete local ring. Regarding $(\hat X,\hat\Delta)$ as if it were a usual pair $(X,\Delta)$, one has only to construct a model $y\in Y\to x\in X$ for the fixed pair $(X,\Delta)$. The remaining case is when $(X,\Delta)$ is log canonical and has the smallest lc centre $C$ of dimension one, which is smooth. After a reduction using the ACC for canonical thresholds \cite{HLLaxv}, we inductively build an infinite sequence of divisorial contractions $E_{i+1}\subset X_{i+1}\to x_i\in X_i$ contracting the divisor $E_{i+1}$ to the point $x_i$ over $x$ that lies on the strict transform $C_i$ of $C$, where $x_0\in X_0$ is $x\in X$, as in Proposition~\ref{prp:sequence}. It suffices to find an index $i$ such that $X_i$ is smooth at $x_i$.

We have a classification of threefold divisorial contractions mainly due to the author \cite{K01}, \cite{K02}, \cite{K03}, \cite{K05}. One can derive from the classification that $X_i$ is Gorenstein at $x_i$ for infinitely many $i$ as in Corollary~\ref{crl:nG}. Hence we consider the composite $x_j\in X_j\to x_i\in X_i$ where $i<j$ such that both $x_i\in X_i$ and $x_j\in X_j$ are Gorenstein. Define an invariant $m_i$ to be the minimum of the integer which equals the order along $E_j$ of the maximal ideal in $\sO_{X_i}$ defining $x_i$ for infinitely many $j$. By definition, $m_i$ does not increase and thus we may assume $m_i$ to be constant. Then we study an invariant $l_i$ originally introduced by Mori \cite{Mo88} for the existence of threefold flips. It is defined to be the length of the cokernel of the natural map $\bigwedge^2\sI_i/\sI_i^{(2)}\to\omega_{X_i}\otimes\omega_{C_i}^{-1}$, where $\sI_i$ denotes the ideal sheaf in $\sO_{X_i}$ defining $C_i$. By a careful application of the classification, we conclude as in Propositions~\ref{prp:G-G} and~\ref{prp:G-nG-G} that $l_i$ decreases strictly until we attain a smooth point $x_k$ at some index $k$.

\section{Preliminaries}\label{sct:prelim}
We follow the notation and basic definitions in our preceding papers \cite{K21}, \cite{Kaxv}. We work over an algebraically closed field $k$ of characteristic zero. A \textit{variety} means an integral separated scheme of finite type over $\Spec k$. An ideal sheaf on a variety is assumed to be coherent. The germ is considered at a closed point. A \textit{contraction} means a projective morphism of normal varieties with connected fibres. The \textit{round-down} $\rd{r}$ of a real number $r$ is the greatest integer less than or equal to $r$ whilst the \textit{round-up} $\ru{r}$ is defined as $\ru{r}=-\rd{-r}$.

Let $X$ be a variety and let $Z$ be a closed subvariety of $X$. The \textit{order} $\ord_Z\sI$ along $Z$ of an ideal sheaf $\sI$ in $\sO_X$ is the supremum of the integers $\nu$ such that $\sI\sO_{X,\eta}\subset\fm_\eta^\nu$ for the generic point $\eta$ of $Z$ and the maximal ideal $\fm_\eta$ in $\sO_{X,\eta}$. For a prime divisor $E$ on a normal variety $Y$ equipped with a birational morphism to $X$, we write $\ord_E\sI$ for $\ord_E\sI\sO_Y$. For a function $f$ in $\sO_X$, we write $\ord_Zf$ for $\ord_Zf\sO_X$. When $X$ is normal, the \textit{order} $\ord_ZD$ of an effective $\bQ$-Cartier divisor $D$ on $X$ is defined as $r^{-1}\ord_Z\sO_X(-rD)$ by a positive integer $r$ such that $rD$ is Cartier. The notion of $\ord_ZD$ is linearly extended to $\bR$-Cartier $\bR$-divisors.

A \textit{pair} $(X,\Delta)$ consists of a normal variety $X$ and an effective $\bR$-divisor $\Delta$ on $X$ such that $K_X+\Delta$ is $\bR$-Cartier, in which $\Delta$ is called the \textit{boundary}. Unlike our preceding works, we shall not treat a triple with an $\bR$-ideal except for Theorem~\ref{thm:adic} and the proofs of Theorem~\ref{thm:success} and Proposition~\ref{prp:sequence}. A divisor \textit{over} $X$ means a prime divisor $E$ on some normal variety $Y$ equipped with a birational morphism $\pi\colon Y\to X$. The closure of the image $\pi(E)$ is called the \textit{centre} in $X$ of $E$ and denoted by $c_X(E)$. Two divisors over $X$ are usually identified if they define the same valuation on the function field of $X$. The \textit{log discrepancy} of $E$ with respect to the pair $(X,\Delta)$ is
\[
a_E(X,\Delta)=1+\ord_E(K_Y-\pi^*(K_X+\Delta)).
\]
For the birational morphism $\pi\colon Y\to X$, we say that a pair $(Y,\Gamma)$ is \textit{crepant} to $(X,\Delta)$ if $a_E(Y,\Gamma)=a_E(X,\Delta)$ for all divisors $E$ over $Y$. When the boundary $\Delta$ is zero, in which $X$ is $\bQ$-Gorenstein, we write $a_E(X)$ for $a_E(X,0)$ and the \textit{discrepancy} of $E$ with respect to $X$ is
\[
d_E(X)=a_E(X)-1=\ord_E(K_Y-\pi^*K_X).
\]
The \textit{index} of a $\bQ$-Gorenstein normal variety $X$ is the least positive integer $r$ such that $rK_X$ is Cartier.

Let $\eta$ be a scheme-theoretic point in $X$ and let $Z$ denote the closure of $\{\eta\}$ in $X$. The \textit{minimal log discrepancy} of $(X,\Delta)$ at $\eta$ is
\[
\mld_\eta(X,\Delta)=\inf\{a_E(X,\Delta)\mid\textrm{$E$ a divisor over $X$},\ c_X(E)=Z\}.
\]
It is either a non-negative real number or minus infinity. We say that a divisor $E$ over $X$ \textit{computes} $\mld_\eta(X,\Delta)$ if $c_X(E)=Z$ and $a_E(X,\Delta)=\mld_\eta(X,\Delta)$ (or $a_E(X,\Delta)<0$ when $\mld_\eta(X,\Delta)=-\infty$). It often suffices to deal with the case when $Z$ is a closed point by the relation
\[
\mld_\eta(X,\Delta)=\mld_x(X,\Delta)-\dim Z
\]
for the general closed point $x$ in $Z$. We also define the \textit{minimal log discrepancy} $\mld_W(X,\Delta)$ of $(X,\Delta)$ in a closed subset $W$ of $X$ as the infimum of $a_E(X,\Delta)$ for the divisors $E$ over $X$ such that $c_X(E)\subset W$.

The pair $(X,\Delta)$ is said to be \textit{terminal} (resp.\ \textit{canonical}, \textit{purely log terminal} (\textit{plt})) if $a_E(X,\Delta)>1$ (resp.\ $\ge1$, $>0$) for all divisors $E$ exceptional over $X$. It is said to be \textit{Kawamata log terminal} (\textit{klt}) (resp.\ \textit{log canonical} (\textit{lc})) if $a_E(X,\Delta)>0$ (resp.\ $\ge0$) for all divisors $E$ over $X$. It is said to be \textit{divisorially log terminal} (\textit{dlt}) if it is log canonical and admits a log resolution $X'\to X$ such that $a_E(X,\Delta)>0$ for all divisors $E$ on $X'$ exceptional over $X$, where a \textit{log resolution} means a birational contraction from a smooth variety such that the exceptional locus is a divisor and such that the union of the exceptional locus and the support of the strict transform of the boundary is simple normal crossing. When $\Delta$ is zero, we say that $X$ is terminal, canonical and so forth. In this case, the notions of klt, plt and dlt singularities coincide and we simply say that $X$ is \textit{log terminal} (\textit{lt}). When $(X,\Delta)$ is log canonical, the centre $c_X(E)$ of a divisor $E$ over $X$ such that $a_E(X,\Delta)=0$ is called an \textit{lc centre}. Considered on a germ, an lc centre contained in all lc centres is called the \textit{smallest lc centre}.

Let $o\in\bA^d$ be the germ at origin of the affine space with coordinates $x_1,\ldots,x_d$. The notation $\bA^d/\bZ_r(a_1,\ldots,a_d)$ stands for the quotient of $\bA^d$ by the cyclic group $\bZ_r$ of order $r$ whose generator sends $x_i$ to $\zeta^{a_i}x_i$ for a primitive $r$-th root $\zeta$ of unity. A singularity \'etale to $o\in A=\bA^d/\bZ_r(a_1,\ldots,a_d)$ is called a \textit{cyclic quotient singularity} of \textit{type} $\frac{1}{r}(a_1,\ldots,a_d)$. When $k$ is the field $\bC$ of complex numbers, we use the notation $o\in\fD^d/\bZ_r(a_1,\ldots,a_d)$ for the analytic germ of $o\in A$.

The quotient $A$ is the toric variety $T_N(\Delta)$ for the lattice $N=\bZ^d+\bZ v$ with $v=\frac{1}{r}(a_1,\ldots,a_d)$ and the standard fan $\Delta$. For a primitive element $e=\frac{1}{r}(w_1,\ldots,w_d)$ in $N$ with all $w_i$ positive, we define the \textit{weighted blow-up} $B\to A$ with $\wt(x_1,\ldots,x_d)=\frac{1}{r}(w_1,\ldots,w_d)$ by adding the ray generated by $e$. If $N=\bZ^d+\bZ e$, then $B$ is covered by the affine charts $(x_i\neq0)\simeq\bA^d/\bZ_{w_i}(w_1,\ldots,w_{i-1},-r,w_{i+1},\ldots,w_d)$ for $1\le i\le d$. For a closed subvariety $o\in X$ of the germ $o\in A$, the induced morphism $X_B\to X$ from the strict transform in $B$ is the \textit{weighted blow-up} of $o\in X\subset A$ with the weights above. For a germ $v\in V$ equipped with an \'etale morphism to $o\in X$, the \textit{weighted blow-up} of $V$ is defined as the base change $X_B\times_XV\to V$. See \cite[section~2.2]{K24} for details.

\section{Reduction to a fixed terminal threefold}
The purpose of this section is to reduce the theorems in the introduction to the following boundedness on a fixed terminal singularity.

\begin{theorem}\label{thm:reduction}
Let $x\in X$ be the germ of a $\bQ$-factorial terminal threefold and fix a positive integer $n$. Let $\{\Delta_i\}_{i\in\bN}$ be an infinite sequence of effective $\bQ$-divisors on $X$ such that $n\Delta_i$ is a Cartier divisor. Then there exists a positive integer $l$ such that for infinitely many indices $i$, there exists a divisor $E_i$ over $X$ which computes $\mld_x(X,\Delta_i)$ and has $a_{E_i}(X)\le l$.
\end{theorem}

\begin{lemma}\label{lem:dlt}
Let $X$ be a normal quasi-projective threefold. Then there exists a contraction $\mu\colon X'\to X$ from a $\bQ$-factorial normal threefold such that the divisorial part $S'$ of the exceptional locus forms a dlt pair $(X',S')$ and such that for every $\bR$-divisor $\Delta$ on $X$ forming a pair $(X,\Delta)$, the $\bR$-divisor $\Delta'$ on $X'$ defined by the equality $K_{X'}+\Delta'=\mu^*(K_X+\Delta)$ satisfies $S'\le\Delta'$.
\end{lemma}

\begin{proof}
Take a log resolution $Y\to X$ and let $S_Y$ denote the exceptional locus. Running the $(K_Y+S_Y)$-MMP over $X$ for the dlt pair $(Y,S_Y)$ by \cite{Sh96}, we produce a contraction $\mu\colon X'\to X$ such that $K_{X'}+S'$ is nef over $X$, where $S'$ is the strict transform of $S_Y$. For a pair $(X,\Delta)$ in the statement, $S'-\Delta'=(K_{X'}+S')-\mu^*(K_X+\Delta)$ is nef over $X$ and hence $S'-\Delta'\le0$ from the negativity lemma \cite[lemma~2.19]{Ko+92}.
\end{proof}

The Cartier index is bounded on a log terminal variety. The reader might compare it with Lemma~\ref{lem:class}.

\begin{theorem}[Greb--Kebekus--Peternell \cite{GKP16}]\label{thm:GKP}
Let $X$ be a normal variety which admits a klt pair $(X,\Delta)$. Then there exists a positive integer $r$ such that for every $\bQ$-Cartier divisor $D$ on $X$, the multiple $rD$ is Cartier.
\end{theorem}

\begin{proof}
The assertion over $\bC$ is in \cite[remark~1.11]{GKP16}. Over $k$, we want to bound the local Cartier index at $x_i$ of $D_i$ for an arbitrary sequence $\{(D_i,x_i)\}_{i\in\bN}$ of pairs of a $\bQ$-Cartier divisor on $X$ and a closed point in $X$. There exists a subfield $k_0$ of $k$ countably generated over $\bQ$ such that $(X,\Delta)$ and all $D_i$ and $x_i$ are defined over $k_0$. Extend the objects over $k_0$ by an embedding of $k_0$ into $\bC$ and apply the assertion over $\bC$. Note that $rD_i$ is Cartier at $x_i$ if and only if $\sO_X(rD_i)\otimes k(x_i)\simeq k$ for the residue field $k(x_i)$ at $x_i$.
\end{proof}

\begin{lemma}\label{lem:reduction}
Theorems~\textup{\ref{thm:main}} to~\textup{\ref{thm:nakamura}} follow from Theorem~\textup{\ref{thm:reduction}}.
\end{lemma}

\begin{proof}
Note that for the theorems, one has only to consider a closed point as $\eta$ by the relation $\mld_\eta(X,\Delta)=\mld_x(X,\Delta)-\dim Z$ mentioned in Section~\ref{sct:prelim}. We may assume the fixed threefold $X$ to be quasi-projective.

Take $\mu\colon X'\to X$ as in Lemma~\ref{lem:dlt}. Consider an arbitrary pair $(X,\Delta)$ and define $\Delta'$ as in the lemma so that $K_{X'}+\Delta'=\mu^*(K_X+\Delta)$. Given $\eta\in X$, there exists a scheme-theoretic point $\eta'\in X'$ over $\eta$ such that $\mld_\eta(X,\Delta)=\mld_{\eta'}(X',\Delta')$. It follows from the inequality $S'\le\Delta'$ that if $(X,\Delta)$ is log canonical at $\eta$, then $\Delta'=\mu_*^{-1}\Delta+S'$ at $\eta'$ for the strict transform $\mu_*^{-1}\Delta$ of $\Delta$ and thus $\mld_\eta(X,\Delta)=\mld_{\eta'}(X',\mu_*^{-1}\Delta+S')$. Hence by replacing $X$ by $X'$ and adding one to $I$, we may assume that $X$ is $\bQ$-factorial and log terminal. Note that $\mu$ is a small $\bQ$-factorialisation if $X$ is log terminal.

Then we take a $\bQ$-factorial terminalisation $\pi\colon Y\to X$, that is, $Y$ is $\bQ$-factorial and admits a terminal pair $(Y,B)$ such that $K_Y+B=\pi^*K_X$. In particular, $Y$ is terminal. Take an integer $r$ supplied for $X$ by Theorem~\ref{thm:GKP}. For every pair $(X,\Delta)$ with $\Delta\in I$, the pair $(Y,B+\pi^*\Delta)$ is crepant to it and the coefficients in $\pi^*\Delta$ belong to the set $r^{-1}I_+=\{r^{-1}\sum_i\delta_i\mid\delta_i\in I\}$ which satisfies the DCC\@. Replacing $X$ by $Y$ and $I$ by the union $K\cup r^{-1}I_+\cup(K+r^{-1}I_+)$ for the set $K$ of the coefficients in $B$, we may assume that $X$ is terminal.

The $\bQ$-factorial terminal threefold $X$ is smooth outside finitely many points. Since we know the four statements on smooth threefolds \cite{K24}, the statements on $X$ follow from those on the germ of $X$ at each of the finitely many points. Hence we may assume the threefold to be the fixed germ $x\in X$ of a $\bQ$-factorial terminal threefold.

We have confirmed the equivalence of the ACC for minimal log discrepancies, the ACC for $a$-lc thresholds, the uniform $\fm$-adic semi-continuity and Nakamura's boundedness for pairs with $\bR$-ideals on a fixed germ \cite[theorem~4.6]{K21}, \cite[proposition~8.1]{Kaxv}. Let $\fm$ denote the maximal ideal in $\sO_X$ defining $x$. Take an integer $r$ as in Theorem~\ref{thm:GKP}. For an $\bR$-divisor $\Delta=\sum_j\delta_j\Delta_j$ on $X$ with $\delta_j\in I$, the minimal log discrepancy $\mld_x(X,\Delta)$ coincides with $\mld_x(X,\fa)$ for the $\bR$-ideal $\fa=\prod_j\fa_j^{\delta_j/r}$ with $\fm$-primary ideals $\fa_j=\sO_X(-r\Delta_j)+\fm^l$ for a sufficiently large integer $l$, where $r\Delta_j$ is Cartier. Through the equivalence mentioned above together with the reduction \cite[lemma~4.11]{K21} to the case of rational coefficients, we complete the reduction of Theorems~\ref{thm:main} to~\ref{thm:nakamura} to Theorem~\ref{thm:reduction}. Note that the boundedness in Theorem~\ref{thm:reduction} is equivalent to that for a sequence of $\bQ$-ideals, since $\mld_x(X,\fa^{1/n})$ for an $\fm$-primary ideal $\fa$ in $\sO_X$ equals $\mld_x(X,n^{-1}A)$ for the divisor $A$ defined by the general member of $\fa$.
\end{proof}

We used the log terminal assumption in Theorem~\ref{thm:nakamura} in controlling the divisors that compute $\mld_\eta(X,\Delta)=-\infty$. Even if $X$ is not log terminal, the boundedness holds for $\Delta$ such that $\mld_\eta(X,\Delta)\ge0$.

\section{The generic limit of functions}\label{sct:limit}
We adopt the notion of the generic limit of functions in order to stay with $\bQ$-divisors. In fact, the generic limit was originally constructed for functions \cite{Koaxv}. We retain our style in \cite{K21} and \cite{Kaxv} by minimal modification. We review the theory briefly and refer the reader to \cite[sections~3,~4]{K21} for details.

Fix the germ $x\in X$ of a variety and let $\fm$ denote the maximal ideal in $\sO_X$ defining $x$. Consider an infinite sequence $\cS=\{f_i\}_{i\in\bN}$ of functions in $\sO_X$. Using Hilbert schemes, one can construct a \textit{family} $\cF=(Z_l,f(l),N_l,s_l,t_l)_{l\ge l_0}$ \textit{of approximations} of $\cS$ \cite[definition~3.1]{K21}. Here $f(l)$ is a function in $\sO_{X\times Z_l}$ compatible with a dominant morphism $t_l\colon Z_{l+1}\to Z_l$ of varieties in the sense that $t_l^*f(l)-f(l+1)$ belongs to $\fm^l\sO_{X\times Z_{l+1}}$, and $s_l\colon N_l\to Z_l(k)$ is a map from an infinite subset of $\bN$ to the set of $k$-points in $Z_l$ with dense image, satisfying $N_{l+1}\subset N_l$ and $t_l\circ s_{l+1}=s_l|_{N_{l+1}}$. The function $f(l)_i=f(l)|_{X\times s_l(i)}$ in $\sO_X$ given by the point $s_l(i)\in Z_l$ approximates $f_i$ in the sense that $f(l)_i-f_i$ belongs to $\fm^l$.

Define the field $K$ to be the algebraic closure of the direct limit $\varinjlim_lk(Z_l)$ of the function field $k(Z_l)$ of $Z_l$. Unlike our previous formulation, we take the algebraic closure as a matter of convenience. Let $\hat X$ be the spectrum of the completion of the local ring $\sO_{X,x}\otimes_kK$ and let $\hat x$ denote the closed point in $\hat X$. The \textit{generic limit} of $\cS$ with respect to $\cF$ is the function $\hat f$ in $\sO_{\hat X}$ defined by the inverse limit
\[
\hat f=\varprojlim_lf(l).
\]

Suppose that $X$ is normal and fix a positive integer $n$. Each function $f_i$ defines an effective Cartier divisor $D_i$ on $X$ whilst $\hat f$ defines an effective Cartier divisor $\hat D$ on $\hat X$. The $\bQ$-divisor $\hat\Delta=n^{-1}\hat D$ is the \textit{generic limit} of the sequence of $\bQ$-divisors $\Delta_i=n^{-1}D_i$ on $X$. The notions of singularities in Section~\ref{sct:prelim} make sense on $\hat X$ \cite{dFEM11}.

The regular morphism $\hat x\in\hat X\to x\in X$ factors through the base change $x_K\in X_K$ of $x\in X$ by $\Spec K\to\Spec k$. The intermediate germ $x_K\in X_K$ is defined over an algebraically closed field $K$ and thus one can treat it without any special care. On the other hand, even if there exists a good birational model $Y_K\to X_K$, it does not always \textit{descend} to a model $Y_l\to X\times Z_l$ in the sense that $Y_K=Y_l\times_{Z_l}\Spec K$ but only to a model $Y_l'\to X\times Z_l'$ after a quasi-finite extension $Z_l'\to Z_l$. For this reason, we need to modify the notion of a \textit{subfamily} $\cF'=(Z'_l,f'(l),N'_l,s'_l,t'_l)_{l\ge l'_0}$ in \cite[definition~3.5]{K21} in such a manner that the compatible morphism $Z'_l\to Z_l$ is only a quasi-finite dominant morphism. We shall often replace $\cF$ by a subfamily and accordingly replace $\cS$ by an infinite subsequence.

Suppose that the normal variety $X$ is $\bQ$-Gorenstein. The ACC for minimal log discrepancies on the fixed germ $x\in X$ is equivalent to the equality $\mld_{\hat x}(\hat X,\hat\Delta)=\mld_x(X,\Delta_i)$.

\begin{lemma}[{\cite[lemma~4.7]{K21}}]\label{lem:limit}
If the equality $\mld_{\hat x}(\hat X,\hat\Delta)=\mld_x(X,\Delta_i)$ holds for all $i\in N_{l_0}$, then after replacement of $\cF$ by a subfamily, there exists a rational number $l$ such that for every $i\in N_{l_0}$ there exists a divisor $E_i$ over $X$ which computes $\mld_x(X,\Delta_i)$ and has $a_{E_i}(X)=l$.
\end{lemma}

The achievement on log canonicity stated as Theorem~\ref{thm:dFEM} settles the above equality when $\mld_{\hat x}(\hat X,\hat\Delta)$ is not positive. The equality also holds when $(\hat X,\hat\Delta)$ is klt \cite[theorem~5.1]{K15}, the proof of which works even if $X$ is singular.

\begin{theorem}[de Fernex--Ein--Musta\c{t}\u{a} \cite{dFEM10}, \cite{dFEM11}]\label{thm:dFEM}
If $(\hat X,\hat\Delta)$ is log canonical, then so is $(X,\Delta_i)$ for all $i\in N_{l_0}$ after replacement of $\cF$.
\end{theorem}

Henceforth we shall assume $x\in X$ to be the germ of a $\bQ$-factorial terminal threefold. The germ $x_K\in X_K$ is also $\bQ$-factorial as well as terminal. Indeed, take a small $\bQ$-factorialisation $X_K'\to X_K$, which descends to a small projective morphism $X_l'\to X\times Z_l$ after replacement of $\cF$. The fibre $X_{lz}'\to X$ at the general point $z$ in $Z_l$ is a small contraction to a $\bQ$-factorial threefold and thus it is an isomorphism. It follows that the geometric generic fibre $X_K'\to X_K$ is also an isomorphism.

In our setting, we have the equality in Lemma~\ref{lem:limit} unless $(\hat X,\hat\Delta)$ is log canonical and has the smallest lc centre $\hat C$ of dimension one, which is regular. Indeed, as remarked after that lemma, it suffices to discuss the case when $(\hat X,\hat\Delta)$ is not klt with $\mld_{\hat x}(\hat X,\hat\Delta)>0$. Then, following the argument in \cite[section~1]{Ka97} with the vanishing theorem due to Murayama \cite{Muapp}, or adjusting the proof of \cite[theorem~1.2]{K15}, one can prove that $(\hat X,\hat\Delta)$ has the smallest lc centre and that it is normal. Since $\mld_{\hat x}(\hat X,\hat\Delta)$ is positive, the smallest lc centre is either a regular curve or a normal surface. If it is a surface, then the equality in the lemma is derived from a minor modification of \cite[theorem~5.3]{K15}. In fact, the proof of Theorem~\ref{thm:success} works in disregard of $\hat C$.

We extend the result \cite[proposition~6.1]{K15} to a terminal threefold.

\begin{lemma}\label{lem:LM}
Notation and assumptions as above, where $x\in X$ is the germ of a $\bQ$-factorial terminal threefold. Suppose that $(\hat X,\hat\Delta)$ is log canonical and has the smallest lc centre $\hat C$ of dimension one. Then $\mld_{\hat x}(\hat X,\hat\Delta)\le1$.
\end{lemma}

\begin{proof}
Thanks to Lyu and Murayama \cite{LMaxv}, we have a $\bQ$-factorial terminalisation $\hat\pi\colon\hat Y\to\hat X$ of the klt pair $(\hat X,(1-\varepsilon)\hat\Delta)$ for a small positive rational number $\varepsilon$. It extracts a divisor $\hat F$ which computes $\mld_{\hat\eta}(\hat X,\hat\Delta)=0$ at the generic point $\hat\eta$ of $\hat C$. Write $K_{\hat Y}+\hat F+\hat B=\hat\pi^*(K_{\hat X}+\hat\Delta)$ with an effective $\bQ$-divisor $\hat B$. For a curve $\hat l$ in $\hat F$ mapped to $\hat x$, the terminal threefold $\hat Y$ is regular at the generic point $\hat\eta_l$ of $\hat l$ and the divisor $\hat E$ obtained at $\hat\eta_l$ by the blow-up of $\hat Y$ along $\hat l$ satisfies
\[
\mld_{\hat x}(\hat X,\hat\Delta)\le a_{\hat E}(\hat X,\hat\Delta)=a_{\hat E}(\hat Y,\hat F+\hat B)=2-\ord_{\hat E}(\hat F+\hat B)\le1.
\qedhere
\]
\end{proof}

In the remaining case, we shall construct a contraction $\pi_K\colon Y_K\to X_K$ from a $\bQ$-factorial terminal threefold which is isomorphic outside $x_K$. Let $\hat\pi\colon\hat Y\to\hat X$ denote the base change and let $\hat y\in\hat Y$ denote the point above $\hat x$ that lies on the strict transform of $\hat C$. The point $\hat y$ descends to a point $y_K\in Y_K$. Replacing $\cF$, $\pi_K$ together with $y_K$ descends to a contraction $\pi_l\colon Y_l\to X\times Z_l$ endowed with a section $y_l\subset Y_l$ of $Y_l\to Z_l$ such that for all $i$ the fibre $\pi_i\colon Y_i\to X$ at $s_l(i)\in Z_l$ is a contraction from a terminal threefold endowed with a point $y_i\in Y_i$. Every exceptional prime divisor $E_K$ on $Y_K$ also descends to a divisor $E_l$ on $Y_l$. We may assume descent to be compatible with base change, so that $\pi_{l+1}$ is the base change of $\pi_l$ by $t_l$ and hence the fibre $\pi_i$ is independent of $l$ as far as $i\in N_l$. We may further assume that $a_{\hat E}(\hat X)=a_{E_i}(X)$ and $\ord_{\hat E}\hat\Delta=\ord_{E_i}\Delta_i$ for the divisor $\hat E=E_K\times_{X_K}\hat X$ on $\hat Y$ and the divisor $E_i=E_l\times_{Z_l}s_l(i)$ on $Y_i$, which is prime as $K$ is algebraically closed.

One should note that $\hat X$ may not be $\bQ$-factorial as observed in \cite[example~1.2.3]{K24}. In contrast to \cite[proposition~A.7]{K15}, a contraction $\hat Y\to\hat X$ isomorphic outside $\hat x$ may not descend to $Y_K\to X_K$.

\begin{theorem}\label{thm:success}
Notation and assumptions as above, where $x\in X$ is the germ of a $\bQ$-factorial terminal threefold and $(\hat X,\hat\Delta)$ is log canonical and has the smallest lc centre $\hat C$ of dimension one. Given a contraction $\pi_K\colon Y_K\to X_K$ as above, we define the $\bQ$-divisors $\hat\Gamma$ on $\hat Y$ and $\Gamma_i$ on $Y_i$ by the equalities $K_{\hat Y}+\hat\Gamma=\hat\pi^*(K_{\hat X}+\hat\Delta)$ and $K_{Y_i}+\Gamma_i=\pi_i^*(K_X+\Delta_i)$ respectively. Suppose that $\hat\Gamma$ is effective and thus so is $\Gamma_i$. Then after replacement of $\cF$ by a subfamily, $\mld_x(X,\Delta_i)$ equals $\mld_{\hat x}(\hat X,\hat\Delta)$ or $\mld_{y_i}(Y_i,\Gamma_i)$ for all $i$.
\end{theorem}

\begin{proof}
The proof is the same as that of \cite[theorem~4.10]{Kaxv} except that $\hat X$ may not be $\bQ$-factorial. We shall explain the way of reformulation, along which one can write down the complete proof following the proof in \cite{Kaxv} verbatim.

We shall use the notion of a triple where an $\bR$-ideal is added to a pair. Recall the notation $\hat D=n\hat\Delta$ and $D_i=n\Delta_i$. We also write $D(l)$ for the Cartier divisor on $X\times Z_l$ defined by the function $f(l)$. Since $\hat C$ is the smallest lc centre of $(\hat X,\hat\Delta)$, $\mld_{\hat x}(\hat X,\hat\Delta)$ is positive. Since $Y_K$ is $\bQ$-factorial, the exceptional locus $\hat F$ of $\hat\pi$ consists of $\bQ$-Cartier prime divisors. By Theorem~\ref{thm:dFEM}, $(X,\Delta_i)$ is log canonical for all $i$ after replacement of $\cF$. Replacing $\cF$, it suffices to prove the inequality $\mld_{\hat x}(\hat X,\Delta)\le a_E(X,\Delta_i)$ for all $i\in N_{l_0}$ and all divisors $E$ over $X$ such that $c_X(E)=x$ and $c_{Y_i}(E)\neq y_i$.

The reduced part $\hat S=\rd{\hat\Delta}$ of $\hat\Delta$ (which may be zero) is not necessarily $\bQ$-Cartier. For this reason, we take the general hyperplane section $\hat H$ in $\hat X$ containing the support of $\hat\Delta-\hat S$ and fix a positive rational number $t$ such that $(\hat Y,\hat\Gamma+t\hat\pi^*\hat H)$ is log canonical outside the strict transform $\hat C_Y$ of $\hat C$. Since $(\hat Y,\hat\Gamma)$ is plt outside $\hat C_Y$, the strict transform $\hat S_Y$ of $\hat S$ is normal outside $\hat C_Y$. There exists a log resolution $\hat V\to\hat Y$ of $(\hat Y,\hat F)$ isomorphic outside $\hat F$ such that the sum $\hat G+\hat S_V$ of the exceptional locus $\hat G=\sum_j\hat G_j$ of $\hat V\to\hat X$ and the strict transform $\hat S_V$ of $\hat S$ is simple normal crossing outside the inverse image $\hat N$ of $\hat C_Y$. Let $\hat L$ denote the restriction to $\hat G$ of $\hat S_V$ and let $\hat M$ denote the restriction to $\hat G$ of the strict transform of the support of $\hat\Delta-\hat S$. We may assume that $\hat M\setminus\hat N$ is contained in the union of the divisors $\hat G_j$ such that $t\ord_{\hat G_j}\hat H\ge\mld_{\hat x}(\hat X,\hat\Delta)$. We take an $\hat\fm$-primary ideal $\sfd=\sO_{\hat X}(-\hat H)+\hat\fm^h$ in $\sO_{\hat X}$ by the maximum $h$ of $\ord_{\hat G_j}\hat H$, where $\hat\fm$ denotes the maximal ideal in $\sO_{\hat X}$. Then $\ord_{\hat G_j}\sfd=\ord_{\hat G_j}\hat H$ and the triple $(\hat Y,\hat\Gamma,(\sfd\sO_{\hat Y})^t)$ is log canonical outside $\hat C_Y$.

We can assume that $\hat V\to\hat Y$ descends to a resolution $V_l\to Y_l$ for all $l\ge l_0$ with the following good properties. The loci $\hat G_j$, $\hat L$ and $\hat M$ descend to flat families $G_{jl}$, $L_l$ and $M_l$ in $V_l$, whilst $\hat N$ descends to the inverse image $N_l$ in $V_l$ of $y_l$. The ideal $\sfd$ descends to an ideal sheaf $\fd_l$ in $\sO_{X\times Z_l}$. The divisor $G_l=\sum_jG_{jl}$ is simple normal crossing and all strata as well as the restrictions to $L_l\setminus N_l$ are smooth over $Z_l$. We express the fibre at $z\in Z_l$ by adding the subscript $z$ such as $G_{jl,z}=G_{jl}\times_{Z_l}z$ and $\fd_{lz}=\fd_l\sO_{X\times z}$. If the fibre at $z=s_l(i)$ is independent of $l$ as far as $i\in N_l$, then we express it by using the subscript $i$ instead of $lz$ such as $V_i=V_{lz}$, $G_{ji}=G_{jl,z}$ and $\fd_i=\fd_{lz}$ With this notation, $\ord_{G_{jl,z}}D(l)_z$ and $\ord_{G_{jl,z}}\fd_{lz}$ are constant on $Z_l$, with $\ord_{G_{jl,z}}D(l)_z<l$, and the equalities
\begin{itemize}
\item
$\sO_{V_l}(-D_V(l))\sO_{\hat V}+\sfI=\sO_{\hat V}(-\hat D_V)+\sfI$ with $\sfI=\sO_{\hat V}(-(n+1)\hat G)$,
\item
$\sO_{V_l}(-D_V(l))\sO_{V_i}+\sI_i=\sO_{V_i}(-D_{Vi})+\sI_i$ with $\sI_i=\sO_{V_i}(-(n+1)G_i)$
\end{itemize}
hold for the strict transforms $D_V(l)$, $\hat D_V$, $D_{Vi}$ of $D(l)$, $\hat D$, $D_i$. The restriction to $G_i$ of $D_{Vi}$ is supported in $L_i\cup M_i$, and the pair $(V_i,G_i+\Delta_{Vi})$ for $\Delta_{Vi}=n^{-1}D_{Vi}$ is log canonical outside $M_i\cup N_i$ by the same argument as for \cite[lemma~5.5]{K15}.

We shall verify the inequality $\mld_{\hat x}(\hat X,\Delta)\le a_E(X,\Delta_i)$. If $c_{V_i}(E)\not\subset M_i\cup N_i$, then we choose any divisor $G_{ji}$ that contains $c_{V_i}(E)$. It follows that
\[
\mld_{\hat x}(\hat X,\hat\Delta)\le a_{\hat G_j}(\hat X,\hat\Delta)=a_{G_{ji}}(X,\Delta_i)\le a_{G_{ji}}(X,\Delta_i)\ord_EG_{ji}\le a_E(X,\Delta_i),
\]
where the last inequality is derived from the above log canonicity of $(V_i,G_i+\Delta_{Vi})$. If $c_{V_i}(E)\subset M_i$ but $c_{V_i}(E)\not\subset N_i$, then $c_{V_i}(E)$ lies on some divisor $G_{ji}$ that has $\ord_{G_{ji}}\fd_i=\ord_{\hat G_j}\sfd=\ord_{\hat G_j}\hat H\ge t^{-1}\mld_{\hat x}(\hat X,\Delta)$. It follows that
\[
\mld_{\hat x}(\hat X,\hat\Delta)\le t\ord_{G_{ji}}\fd_i\le t\ord_E\fd_i\le a_E(Y_i,\Gamma_i)=a_E(X,\Delta_i),
\]
where the last inequality is derived from the log canonicity of $(Y_i,\Gamma_i,(\fd_i\sO_{Y_i})^t)$ outside $y_i$, which is the replacement of \cite[lemma~4.11]{Kaxv}. For the counterpart of the divisor $\hat Q$, we reformulate the proof by fixing an integer $n_Q>r^2\ord_{\hat Q}\hat D_Y$ for a positive integer $r$ such that the multiple by $r$ of every exceptional prime divisor on $Y_K$ is Cartier. Then we attain the equalities
\begin{itemize}
\item
$\sO_{Y_l}(-rD_Y(l))\sO_{\hat Y}+\sfJ=\sO_{\hat Y}(-r\hat D_Y)+\sfJ$ with $\sfJ=\sO_{\hat Y}(-n_Q\hat F)$,
\item
$\sO_{Y_l}(-rD_Y(l))\sO_{Y_i}+\sJ_i=\sO_{Y_i}(-rD_{Yi})+\sJ_i$ with $\sJ_i=\sO_{Y_i}(-n_QF_i)$
\end{itemize}
for the strict transforms $D_Y(l)$, $\hat D_Y$, $D_{Yi}$ of $D(l)$, $\hat D$, $D_i$ and the exceptional loci $\hat F$, $F_i$ of $\hat\pi$, $\pi_i$.
\end{proof}

\section{Construction of a birational model}
We shall explain how to construct a model $Y_K\to X_K$ to which Theorem~\ref{thm:success} will be applied. The first step amounts to the construction in \cite[section~5]{K21} and uses the ACC for canonical thresholds on threefolds.

\begin{proposition}\label{prp:ct}
Let $x\in X$ be the germ of a $\bQ$-factorial terminal threefold and let $\hat x\in\hat X$ denote the spectrum of the completion of the local ring $\sO_{X,x}$. Consider a log canonical pair $(\hat X,\hat\Delta)$ with a $\bQ$-divisor $\hat\Delta$. Then there exists a contraction $\mu\colon X'\to X$ from a $\bQ$-factorial terminal threefold which is isomorphic outside $x$ such that for the base change $\hat\mu\colon\hat X'\to\hat X$ of $\mu$ and for the exceptional $\bQ$-divisor $\hat B$ on $\hat X'$ defined by the equality $K_{\hat X'}+\hat\Delta'+\hat B=\hat\mu^*(K_{\hat X}+\hat\Delta)$ where $\hat\Delta'$ is the strict transform of $\hat\Delta$,
\begin{itemize}
\item
$\hat B$ is effective and
\item
$\hat\Delta'$ is $\bQ$-Cartier and $\mld_{\hat\mu^{-1}(\hat x)}(\hat X',\hat\Delta')\ge1$.
\end{itemize}
\end{proposition}

\begin{proof}
Since $\hat X$ is $\bQ$-Gorenstein, $\hat\Delta$ is $\bQ$-Cartier. Since $X'$ is $\bQ$-factorial, the exceptional locus of $\hat\mu$ consists of $\bQ$-Cartier prime divisors. In particular, $\hat\Delta'$ is $\bQ$-Cartier.

Write $\hat\Delta=q\hat D$ with a Cartier divisor $\hat D$ and a rational number $0<q<1$. We shall inductively build a sequence $\cdots X_i\to\cdots\to X_0=X$ of contractions between $\bQ$-factorial terminal threefolds isomorphic outside $x$. Suppose that $X_i$ has been constructed. Write the base change $\hat X_i=X_i\times_X\hat X$ and let $\hat\Delta_i$ denote the strict transform in $\hat X_i$ of $\hat\Delta$, which is $\bQ$-Cartier for the same reason as for $\hat\Delta'$. For $j<i$, let $\hat S_j$ denote the set of the exceptional prime divisors of $\hat X_{j+1}\to\hat X_j$.

As far as $\mld_{\hat Z_i}(\hat X_i,\hat\Delta_i)$ is less than one for the inverse image $\hat Z_i$ of $\hat x$, we take a rational number $0<c_i<1$ such that $\mld_{\hat Z_i}(\hat X_i,c_i\hat\Delta_i)=1$. The pair $(\hat X_i,c_i\hat\Delta_i)$ is klt and thus the set $\hat S_i$ of the divisors $\hat E$ over $\hat X_i$ with $c_{\hat X}(\hat E)=\hat x$ and $a_{\hat E}(\hat X_i,c_i\hat\Delta_i)=1$ is finite. It descends to a set $S_i$ of divisors $E$ over $X_i$ by the relation $\hat E=E\times_X\hat X$. We take an $\hat\fm$-primary ideal $\hat\fa=\sO_{\hat X}(-\hat D)+\hat\fm^l$ in $\sO_{\hat X}$ for a large integer $l$ such that $\ord_{\hat E}\hat D=\ord_{\hat E}\hat\fa$ for all $\hat E\in\bigcup_{k\le i}\hat S_k$, where $\hat\fm=\fm\sO_{\hat X}$ with the maximal ideal $\fm$ in $\sO_X$ defining $x$. It descends to an $\fm$-primary ideal $\fa$ in $\sO_X$. Set $\Gamma_i=qG_i$ with the strict transform $G_i$ in $X_i$ of the divisor on $X$ defined by the general member of $\fa$. Then $(X_i,c_i\Gamma_i)$ is canonical and $S_i$ coincides with the set of the divisors $E$ over $X_i$ such that $a_E(X_i,c_i\Gamma_i)=1$. We take a $\bQ$-factorial terminalisation $X_{i+1}\to X_i$ of $(X_i,c_i\Gamma_i)$. The set of the exceptional divisors equals $S_i$. It follows that $(\hat X_{i+1},c_i\hat\Delta_{i+1})$ is crepant to $(\hat X_i,c_i\hat\Delta_i)$ and $\mld_{\hat Z_{i+1}}(\hat X_{i+1},c_i\hat\Delta_{i+1})$ is greater than one.

By construction, the sequence of $c_i$ is strictly increasing. Since $c_i$ is the canonical threshold of $\Gamma_i$ on $X_i$, that is, $\mld_{\eta_i}(X_i,c_i\Gamma_i)=1$ at some scheme-theoretic point $\eta_i$, they form a set satisfying the ACC \cite[theorem~1.7]{HLLaxv}. Hence our construction terminates at some $X_i$ and we may take $X_i$ as $X'$.
\end{proof}

The most substantial part of the proof of Theorem~\ref{thm:reduction} is the existence of a smooth point $x_i\in X_i$ in the following construction. Recall that $d_E(X)$ denotes the discrepancy of $E$ with respect to $X$. The precise definition of a \textit{divisorial contraction} will be provided in Definition~\ref{dfn:divcont}.

\begin{proposition}\label{prp:sequence}
Let $x\in X$ be the germ of a $\bQ$-factorial terminal threefold and let $\hat x\in\hat X$ denote the spectrum of the completion of the local ring $\sO_{X,x}$. Consider an lc but not klt pair $(\hat X,\hat\Delta)$ with a $\bQ$-divisor $\hat\Delta$ such that $\mld_{\hat x}(\hat X,\hat\Delta)=1$ and such that the smallest lc centre of $(\hat X,\hat\Delta)$ is a regular curve $\hat C$. Set $x_0\in X_0$ to be $x\in X$. Then there exists an infinite sequence of divisorial contractions $E_{i+1}\subset X_{i+1}\to x_i\in X_i$ contracting the divisor $E_{i+1}$ to a closed point $x_i$ such that for the base change $\hat E_{i+1}\subset\hat X_{i+1}\to\hat x_i\in\hat X_i$ by $\hat X\to X$ and for the strict transforms $\hat\Delta_i$ and $\hat C_i$ in $\hat X_i$ of $\hat\Delta$ and $\hat C$ respectively,
\begin{itemize}
\item
$\hat x_i=\hat C_i\cap\hat E_i$ if $i\ge1$,
\item
$(\hat X_{i+1},\hat\Delta_{i+1})$ is crepant to $(\hat X_i,\hat\Delta_i)$ and
\item
$d_{E_{i+1}}(X_i)/\ord_{E_{i+1}}\fm_i\le d_{E_j}(X_i)/\ord_{E_j}\fm_i$ for all $j>i$, where $\fm_i$ denotes the maximal ideal in $\sO_{X_i}$ defining $x_i$.
\end{itemize}
\end{proposition}

\begin{proof}
We shall use the notion of a triple. We shall build the sequence inductively. Suppose that $X_i$ has been constructed, for which $X_i$ is $\bQ$-factorial and $\hat\Delta_i$ is $\bQ$-Cartier. On the germ $\hat x_i\in\hat X_i$, the strict transform $\hat C_i$ is the smallest lc centre of $(\hat X_i,\hat\Delta_i)$ and it follows from Lemma~\ref{lem:LM} that $\mld_{\hat x_i}(\hat X_i,\hat\Delta_i)\le1$. Hence $\mld_{\hat x_i}(\hat X_i,\hat\Delta_i)=1$ since $\mld_{\hat x_i}(\hat X_i,\hat\Delta_i)\ge\mld_{\hat x}(\hat X,\hat\Delta)=1$.

Consider the minimum $c$ of $d_E(X_i)/\ord_E\fm_i$ for the divisors $E$ over $X_i$ such that $c_{X_i}(E)=x_i$ and $a_{\hat E}(\hat X_i,\hat\Delta_i)=1$ for $\hat E=E\times_X\hat X$. The existence of $c$ follows from the boundedness $\ord_E\fm_i\le t_i^{-1}$ in terms of the log canonical threshold $t_i$ of $\hat\fm_i=\fm_i\sO_{\hat X_i}$ on $(\hat X_i,\hat\Delta_i)$ defined by the equality $\mld_{\hat x_i}(\hat X_i,\hat\Delta_i,\hat\fm_i^{t_i})=0$. Note that $d_E(X_i)=\ord_{\hat E}\hat\Delta_i$ whenever $a_{\hat E}(\hat X_i,\hat\Delta_i)=1$. In particular, for a small positive rational number $\varepsilon$, the triple $(\hat X_i,(1-\varepsilon)\hat\Delta_i,\hat\fm_i^{\varepsilon c})$ retains $\mld_{\hat x_i}(\hat X_i,(1-\varepsilon)\hat\Delta_i,\hat\fm_i^{\varepsilon c})=1$, and every divisor $\hat E=E\times_X\hat X$ that computes this minimal log discrepancy also computes $\mld_{\hat x_i}(\hat X_i,\hat\Delta_i)=1$ and satisfies $d_E(X_i)/\ord_E\fm_i=c$.

Choose one such $E$ and approximate $\hat\Delta_i$ by a $\bQ$-divisor $\Gamma_i$ on $X_i$ in such a way that $a_E(X_i,(1-\varepsilon)\Gamma_i,\fm_i^{\varepsilon c})=1$ by the same argument as in the proof of Proposition~\ref{prp:ct}. That is, for $\hat\Delta=q\hat D$, $\Gamma_i=qG_i$ with the strict transform $G_i$ of the divisor defined by the general member of the contraction $\fa$ of $\hat\fa=\sO_{\hat X}(-\hat D)+\hat\fm^l$, where $\hat\fm$ is the maximal ideal in $\sO_{\hat X}$. Then $(X_i,(1-\varepsilon)\Gamma_i,\fm_i^{\varepsilon c})$ is klt and canonical, and every divisor $E$ that computes $\mld_{x_i}(X_i,(1-\varepsilon)\Gamma_i,\fm_i^{\varepsilon c})=1$ has $a_{\hat E}(\hat X_i,\hat\Delta_i)=1$ and attains the minimum $d_E(X_i)/\ord_E\fm_i=c$. Set an $\fm_i$-primary ideal $\fb_i=\sO_{X_i}(-r_iG_i)+\fm_i^l$ for a large integer $l$, where $r_i$ is a positive integer such that $r_iG_i$ is Cartier. Now take a $\bQ$-factorial terminalisation $W$ of the canonical pair $(X_i,\fb_i^{(1-\varepsilon)q/r_i}\fm_i^{\varepsilon c})$ and run the $K_W$-MMP over $X_i$ to produce a divisorial contraction $E_{i+1}\subset X_{i+1}\to x_i\in X_i$ as the last step. The exceptional divisor $E_{i+1}$ satisfies $a_{\hat E_{i+1}}(\hat X_i,\hat\Delta_i)=1$ and $d_{E_{i+1}}(X_i)/\ord_{E_{i+1}}\fm_i=c$. For all $j>i$, the divisor $E_j$ obtained in the future has $a_{\hat {E_j}}(\hat X_i,\hat\Delta_i)=1$ and hence $c\le d_{E_j}(X_i)/\ord_{E_j}\fm_i$ by the minimality of $c$.
\end{proof}

\begin{theorem}\label{thm:sequence}
In Proposition~\textup{\ref{prp:sequence}}, $X_i$ is smooth at $x_i$ for some index $i$.
\end{theorem}

This will be proved in the last section. Here we shall derive Theorem~\ref{thm:reduction} from it.

\begin{proof}[Proof of Theorems~\textup{\ref{thm:main}} to~\textup{\ref{thm:nakamura}} and~\textup{\ref{thm:reduction}} from Theorem~\textup{\ref{thm:sequence}}]
By virtue of Lemma~\ref{lem:reduction}, it suffices to prove Theorem~\ref{thm:reduction}. Take the function $f_i$ in $\sO_X$ which defines the Cartier divisor $n\Delta_i$ and construct the generic limit $\hat f$ on $\hat x\in \hat X$ of the sequence $\{f_i\}_{i\in\bN}$ following Section~\ref{sct:limit}. We use the notation in Section~\ref{sct:limit} in which $\hat X$ is the spectrum of the completion of the local ring $\sO_{X_K,x_K}$ associated with the germ $x_K\in X_K$ over an algebraically closed extension $K$ of $k$.

Take the pair $(\hat X,\hat\Delta)$ such that $n\hat\Delta$ is the Cartier divisor defined by $\hat f$, that is, $\hat\Delta$ is the generic limit of the sequence of $\Delta_i$. As explained prior to Lemma~\ref{lem:LM}, we may and shall assume that $(\hat X,\hat\Delta)$ is log canonical and has the smallest lc centre $\hat C$ of dimension one, which is regular. From the germ $x_K\in X_K$ and the pair $(\hat X,\hat\Delta)$, we construct by Proposition~\ref{prp:ct} a contraction $\mu\colon X_K'\to X_K$ from a $\bQ$-factorial terminal threefold isomorphic outside $x_K$ and an effective exceptional $\bQ$-divisor $B_K$ on $X_K'$ such that for the base change $\hat\mu\colon \hat X'\to\hat X$ of $\mu$ with $\hat B=B_K\times_{X_K}\hat X$ and the strict transform $\hat\Delta'$ of $\hat\Delta$, the pair $(\hat X',\hat\Delta'+\hat B)$ is crepant to $(\hat X,\hat\Delta)$, $\hat\Delta'$ is $\bQ$-Cartier and $\mld_{\hat\mu^{-1}(\hat x)}(\hat X',\hat\Delta')\ge1$. Let $\hat x'\in\hat X'$ denote the point over $\hat x$ that lies on the strict transform of $\hat C$. Then $\mld_{\hat x'}(\hat X',\hat\Delta')=1$ since $1\le\mld_{\hat\mu^{-1}(\hat x)}(\hat X',\hat\Delta')\le\mld_{\hat x'}(\hat X',\hat\Delta')\le1$ for the same reason as at the beginning of the proof of Proposition~\ref{prp:sequence}. Let $x_K'$ denote the point in $X_K'$ to which $\hat x'$ descends.

Starting with the germ $x_K'\in X_K'$ as $x_{K,0}\in X_{K,0}$ and the pair $(\hat X',\hat\Delta')$, Proposition~\ref{prp:sequence} supplies an infinite sequence of divisorial contractions $E_{K,n+1}\subset X_{K,n+1}\to x_{K,n}\in X_{K,n}$ for $n\in\bN$. It follows from Theorem~\ref{thm:sequence} that there exists an index $n$ such that $X_{K,n}$ is smooth at $x_{K,n}$. Set $y_K\in Y_K$ to be $x_{K,n}\in X_{K,n}$. Define the $\bQ$-divisor $\hat\Gamma$ on $\hat Y=Y_K\times_{X_K}\hat X$ to be the sum of the strict transform of $\hat\Delta'$ and the pull-back of $\hat B$ so that $(\hat Y,\hat\Gamma)$ is crepant to $(\hat X,\hat\Delta)$, and apply Theorem~\ref{thm:success} to the model $\pi_K\colon Y_K\to X_K$. If $\mld_{\hat x}(\hat X,\hat\Delta)=\mld_x(X,\Delta_i)$ for all $i$ after replacement of the family $\cF$, then the existence of the bound $l$ in Theorem~\ref{thm:reduction} follows from Lemma~\ref{lem:limit}.

Otherwise, we may assume that $\mld_x(X,\Delta_i)=\mld_{y_i}(Y_i,\Gamma_i)$ for all $i$ with the notation in Theorem~\ref{thm:success}. Since $Y_K$ is smooth at $y_K$, we may assume that $Y_i$ is smooth at $y_i$. Our preceding work \cite[theorem~1.4]{Kaxv} provides a bound $l'$ such that for every $i$ there exists a divisor $E_i$ over $Y_i$ which computes $\mld_{y_i}(Y_i,\Gamma_i)=\mld_x(X,\Delta_i)$ and has $a_{E_i}(Y_i)\le l'$. For the log canonical threshold $t$ of the maximal ideal $\hat\fm$ on $(\hat X,\hat\Delta)$, we may assume from Theorem~\ref{thm:dFEM} that $(X,\Delta_i,\fm^t)$ is log canonical. Then
\[
t\ord_{E_i}\fm\le a_{E_i}(X,\Delta_i)=\mld_x(X,\Delta_i)\le\mld_{\hat x}(\hat X,\hat\Delta)\le1,
\]
where the middle inequality is an immediate consequence of the construction of the generic limit \cite[remark~4.3]{K21} and the last follows from Lemma~\ref{lem:LM}.

Write $K_{Y_K}=\pi_K^*K_{X_K}+\sum_jd_jF_{jK}$ with discrepancies $d_j=d_{F_{jK}}(X_K)$. Let $d$ be the maximum of $d_j$. The expression of $K_{Y_K}$ induces the expression $K_{Y_i}=\pi_i^*K_X+\sum_jd_jF_{ji}$ on the fibre $Y_i$. Thus
\[
a_{E_i}(X)=a_{E_i}(Y_i)+\sum_jd_j\ord_{E_i}F_{ji}\le l'+d\ord_{E_i}\fm\le l'+dt^{-1},
\]
which yields a bound $l$.
\end{proof}

\section{Threefold divisorial contractions}\label{sct:divcont}
The proof of Theorem~\ref{thm:sequence} requires the explicit study of threefold divisorial contractions mainly due to the author \cite{K01}, \cite{K02}, \cite{K03}, \cite{K05}.

\begin{definition}\label{dfn:divcont}
A \textit{divisorial contraction} $\pi\colon Y\to X$ means a birational contraction between terminal varieties such that the anti-canonical divisor $-K_Y$ is relatively ample and such that the exceptional locus $E$ is a prime divisor. The discrepancy $d_E(X)$ of the exceptional divisor is called the \textit{discrepancy} of $\pi$.
\end{definition}

The discrepancy $d=d_E(X)$ is the positive rational number defined by the equality $K_Y=\pi^*K_X+dE$. It follows that $E$ is $\bQ$-Cartier. If $X$ is $\bQ$-factorial, then so is $Y$. The converse is false such as the blow-up of the ordinary double point given in $\bA^4$ by $x_1x_2+x_3x_4$.

Our object is a threefold divisorial contraction $\pi\colon E\subset Y\to x\in X$ which contracts the divisor $E$ to a point $x$. We shall review some of the results on a classification of $\pi$. The reader may refer to \cite[chapter~3]{K24} for details. In this section, we assume the ground field $k$ to be the field $\bC$ of complex numbers in order to describe $\pi$ analytically.

Let $x\in X$ be the germ of a terminal threefold, which is an isolated singularity. Let $n$ denote the index of $x\in X$. The reader may refer to \cite[chapter~2]{K24} for a classification of $x\in X$. If $n=1$ or equivalently $x\in X$ is Gorenstein, then the general hyperplane section $x\in H$ is a Du Val singularity and $x\in X$ is called a \textit{compound Du Val} (\textit{cDV}) singularity, due to Reid \cite{Re83}. The germ $x\in X$ is said to be of \textit{type c$A_m$}, \textit{c$D_m$} or \textit{c$E_m$} in accordance with the type $A_m$, $D_m$ or $E_m$ of $x\in H$. If $n\ge2$, then $x\in X$ is a cyclic quotient of the index-one cover and thus it is the quotient of a cDV singularity $\tilde x\in\tilde X$ by the cyclic group $\bZ_n$ of order $n$. It is of \textit{type $cA/n$}, \textit{c$Ax/4$}, \textit{c$Ax/2$}, \textit{c$D/3$}, \textit{c$D/2$} or \textit{c$E/2$} as in \cite[p.541]{KM92} in accordance with the type of $\tilde x\in\tilde X$ and the action of $\bZ_n$, due to Mori \cite{Mo85}. In the analytic embedding $\tilde x\in\tilde X\subset\fD^4$ as a hypersurface singularity, the intersection number $(\tilde X\cdot l)_{\tilde x}$ with the axis $l$ of non-free points is called the \textit{axial multiplicity} of $x\in X$. See \cite[definition~2.5.20]{K24} for the explicit description.

The germ $x\in X$ is deformed to a collection of cyclic quotient singularities called \textit{fictitious singularities} from $x\in X$. The set of fictitious singularities is called the \textit{basket} of $x\in X$. Unless $x\in X$ is of type c$Ax/4$, the basket consists of $e$ singularities of index $n$ where $e$ stands for the axial multiplicity. If it is of type c$Ax/4$, then the basket consists of one singularity of index four and $(e-1)/2$ singularities of index two.

We recall the following basic property.

\begin{lemma}[{\cite[lemma~5.1]{Ka88}}]\label{lem:class}
Let $x\in X$ be the germ of a terminal threefold. Then for every $\bQ$-Cartier divisor $D$ on $x\in X$, there exists an integer $l$ such that $D\sim lK_X$.
\end{lemma}

Let $\pi\colon E\subset Y\to x\in X$ be a threefold divisorial contraction which contracts the divisor $E$ to the point $x$. Let $a/n$ denote the discrepancy $d_E(X)$ of $\pi$, that is, $K_Y=\pi^*K_X+(a/n)E$. By deforming the singularities of $Y$ each by each, one obtains the basket $\{y_\iota\ \textrm{of type $\frac{1}{r_\iota}(1,-1,b_\iota)$}\}_{\iota\in I_0}$ of fictitious singularities from $Y$. For each index $\iota\in I_0$, we take an integer $e_\iota$ such that $E\sim e_\iota K_Y$ at $y_\iota$ by virtue of Lemma~\ref{lem:class}. Possibly replacing $b_\iota$ by $r_\iota-b_\iota$, we may and shall assume that $v_\iota=\overline{e_\iota b_\iota}\le r_\iota/2$, where $\bar l$ denotes the residue of $l$ modulo $r_\iota$. Set $J=\{(r_\iota,v_\iota)\}_{\iota\in I}$ indexed by the subset $I$ of $I_0$ consisting of $\iota$ with $\bar e_\iota\neq0$. The following is a numerical classification of $\pi$.

\begin{theorem}[{\cite[theorem~1.1]{K05}}, {\cite[theorem~3.2.2]{K24}}]
The divisorial contraction $\pi$ is of one of the types in Table~\textup{\ref{tbl:numerical}}. The case when $J$ is of form $\{(r_1,1),(r_2,1)\}$ is divided so that $\pi$ is of type \textup{o3} if $J$ comes from two non-Gorenstein points of $Y$ and of type \textup{e2} or \textup{e3} if $J$ comes from one non-Gorenstein point of $Y$.
\begin{table}[ht]\caption{Numerical classification}\label{tbl:numerical}
\begin{tabular}{@{}cl@{}}
\textup{type}&\multicolumn{1}{c}{$J$}\\
\hline
\textup{o1}  &$\emptyset$\\
\textup{o2}  &$(r,1)$\\
\textup{o3}  &$(r_1,1),(r_2,1)$\\
             &\\
             &\\
             &\\
             &\\
             &
\end{tabular}
\qquad
\begin{tabular}{@{}cl@{}}
\textup{type}&\multicolumn{1}{c}{$J$}\\
\hline
\textup{e1}  &$(r,2)$\\
\textup{e2}  &$(r,1),(r,1)$\\
\textup{e3}  &$(2,1),(4,1)$\\
             &\\
\textup{e5}  &$(7,3)$\\
\textup{e6}  &$(8,3)$\\
\textup{e7}  &$(4,2),(r,1)$\\
\textup{e8}  &$(5,2),(2,1)$
\end{tabular}
\qquad
\begin{tabular}{@{}cl@{}}
\textup{type}&\multicolumn{1}{c}{$J$}\\
\hline
\textup{e9}  &$(5,2),(3,1)$\\
\textup{e10} &$(5,2),(4,1)$\\
\textup{e11} &$(6,2),(2,1)$\\
\textup{e12} &$(7,2),(2,1)$\\
\textup{e13} &$(2,1),(2,1),(r,1)$\\
\textup{e14} &$(2,1),(3,1),(3,1)$\\
\textup{e15} &$(2,1),(3,1),(4,1)$\\
\textup{e16} &$(2,1),(3,1),(5,1)$
\end{tabular}
\end{table}
\end{theorem}

We say that $\pi$ is of \textit{ordinary type} if it is of type o1, o2 or o3 whilst we say that it is of \textit{exceptional type} if it is of type e1, \ldots, e15 or e16. Since a singularity $y\in Y$ contributes to $J$ only if $E$ is not Cartier at $y$, we make the following notion.

\begin{definition}\label{dfn:hidden}
A singularity $y\in Y$ is called a \textit{hidden} singularity if $E$ is Cartier at $y$. It follows from the relation $nK_Y\sim aE$ that the index of a hidden singularity is a divisor of the index $n$ of $x\in X$.
\end{definition}

The numerical classification is obtained by computing the dimension $d(l)$ of the complex vector space
\[
V_l=\pi_*\sO_Y(-lE)/\pi_*\sO_Y(-(l+1)E)
\]
for $l\in\bZ$ by means of the singular Riemann--Roch formula \cite[chapter~III]{Re87}. Below we collect some data of $d(l)$ in the Gorenstein case that will be used later.

\begin{lemma}\label{lem:dim}
Suppose that $x\in X$ is Gorenstein, that is, $n=1$.
\begin{enumerate}
\item\label{itm:dim-ord}
If $\pi$ is of ordinary type as in Table~\textup{\ref{tbl:numerical}}, then $d(l)=1+\rd{l/r_1}+\rd{l/r_2}$ for $0\le l<a$, where we set $(r_1,r_2)=(1,r)$ if type \textup{o2} and set $(r_1,r_2)=(1,1)$ if type \textup{o1}. In this case, $a<r_1+r_2$ unless $x\in X$ is smooth.
\item\label{itm:dim-exc}
If $\pi$ is of exceptional type and $a\ge2$, then the pair $(\pi,d(1))$ is $(\textup{e1},1)$, $(\textup{e2},1)$, $(\textup{e3},1)$, $(\textup{e5},0)$ or $(\textup{e9},0)$.
\item\label{itm:dim-e1}
If $\pi$ is of type \textup{e1} and $a=4$, then $d(2)=1$ or $2$.
\end{enumerate}
\end{lemma}

\begin{proof}
In the first item, the value $d(l)$ is found in \cite[lemma~3.3.17]{K24} and the estimate of $a$ follows from \cite[remark~2.1]{K03}. The second item is in \cite[theorem~4.5]{K01}, where $d(1)$ equals $D(2)-1$. In the third item, one can directly compute $d(2)$ from the formula in \cite[lemma~3.3.7]{K24}, where our $d(l)$ equals $d(-l)$ there, as $d(2)=10/r-B_\iota(4)+B_\iota(6)$, which equals two if $r=5$ and equals one if $r>5$.
\end{proof}

\begin{definition}
The \textit{general elephant} of a normal variety $V$ means the general member of the anti-canonical system $\abs{-K_V}$. The \textit{general elephant} of a germ $v\in V$ means the general member of $\abs{-K_V}$ passing through the point $v$.
\end{definition}

Reid pointed out that the general elephant $x\in S$ of an arbitrary terminal threefold singularity $x\in X$ is a Du Val singularity. A geometric classification of $\pi$ is established together with the settlement of the general elephant conjecture as below. In fact, the conjecture holds for an arbitrary threefold birational contraction such that the central fibre is irreducible \cite{KM92}.

\begin{theorem}[{\cite[theorem~1.7]{K03}}, {\cite[theorem~1.5]{K05}}]
Let $\pi\colon E\subset Y\to x\in X$ be a threefold divisorial contraction to the germ $x\in X$ which contracts the divisor $E$ to the point $x$. Then the general elephant $S$ of $Y$ has only Du Val singularities.
\end{theorem}

The image $x\in S_X$ of $S$ is a member of $\abs{-K_X}$ and the induced morphism $S\to S_X$ is crepant. In particular, $x\in S_X$ is also a Du Val singularity. The following notion describes the case when $S_X$ is the general elephant of $x\in X$.

\begin{definition}
Let $\pi\colon E\subset Y\to x\in X$ be a threefold divisorial contraction which contracts the divisor $E$ to a point $x$. Let $n$ denote the index of the germ $x\in X$ and let $\fm$ denote the maximal ideal in $\sO_X$ defining $x$. We say that $\pi$ \textit{keeps general elephants} if
\[
\pi_*\sO_Y(-K_Y)=
\begin{cases}
\fm\sO_X(-K_X)&\textrm{for}\ n=1,\\
\sO_X(-K_X)&\textrm{for}\ n\ge2.
\end{cases}
\]
This means that the strict transform $S$ of the general elephant $S_X$ of the germ $x\in X$ is the general elephant of $Y$ about $E$ \cite[theorem~4.1]{K05}.
\end{definition}

By definition, $\pi$ keeps general elephants whenever $a/n\le1$.

\begin{theorem}\label{thm:keepge}
The divisorial contraction $\pi$ keeps general elephants unless
\begin{enumerate}
\item
$\pi$ is of ordinary type, $a/n>1$ and $x\in X$ is of type c$A$, c$D$ or c$A/n$, or
\item
$\pi$ is of exceptional type, $n=1$ and $(\pi,a)$ is $(\textup{e1},4)$, $(\textup{e1},2)$, $(\textup{e2},2)$ or $(\textup{e3},3)$.
\end{enumerate}
\end{theorem}

\begin{proof}
One has only to discuss the case when $a/n>1$. The case when $\pi$ is of exceptional type is treated in \cite[corollary~3.3.3]{K24}. Suppose that $\pi$ is of ordinary type and that $a/n>1$. By Theorem~\ref{thm:notmin}(\ref{itm:notmin-ord}) below, $x\in X$ is of type c$A$, c$A/n$, c$D$ or c$D/2$. If it is of type c$D/2$, in which $\pi$ is of type o3, then $\pi$ keeps general elephants as shown in \cite[theorem~4.3]{K05}.
\end{proof}

When $x\in X$ is of type c$A$ or c$A/n$, we have an almost complete geometric classification of $\pi$. We also have a list which exhausts the case when $a/n$ is not $1/n$. See \cite[theorems~3.5.5 to~3.5.7]{K24} and the references therein.

\begin{theorem}\label{thm:clsA}
Suppose that $x\in X$ is of type c$A$ or c$A/n$.
\begin{enumerate}
\item\label{itm:clsA-ord}
If $\pi$ is of ordinary type, then there exists an analytic identification
\[
x\in X\simeq o\in(x_1x_2+f(x_3^n,x_4)=0)\subset\fD^4/\bZ_n(1,-1,b,0)
\]
such that $\pi$ is the weighted blow-up with $\wt(x_1,x_2,x_3,x_4)=\frac{1}{n}(r_1,r_2,a,n)$, where $n$ divides $a-br_1$, $an$ divides $r_1+r_2$ and $(a-br_1)/n$ is coprime to $r_1$. Further, $f$ is of weighted order $(r_1+r_2)/n$ with respect to $\wt(x_3,x_4)=(a/n,1)$ and the monomial $x_3^{(r_1+r_2)/a}$ appears in $f$.
\item
If $\pi$ is of exceptional type, then $n=1$ and the triple $(\pi,x,a)$ is $(\textup{e1},\textup{c$A_1$},4)$ or $(\textup{e3},\textup{c$A_2$},3)$.
\end{enumerate}
\end{theorem}

\begin{theorem}\label{thm:notmin}
Suppose that $a/n\neq1/n$.
\begin{enumerate}
\item\label{itm:notmin-ord}
If $\pi$ is of ordinary type, then $x\in X$ is of type $cA$, c$A/n$, c$D$ or c$D/2$. The case when $x\in X$ is of type c$D$ or c$D/2$ occurs only if $\pi$ is of type \textup{o3}.
\item
If $\pi$ is of exceptional type, then it belongs to one of the cases in Table~\textup{\ref{tbl:notmin}}.
\begin{table}[ht]\caption{Exceptional type with $a/n\neq1/n$}\label{tbl:notmin}
\begin{tabular}{@{}clc@{}}
\textup{type}&\multicolumn{1}{c}{$x$}&$a/n$\\
\hline
\textup{e1}  &\textup{c}$A_1$, \textup{c}$D$ &$4$\\
             &\textup{c}$D$                  &$2$\\
             &\textup{c}$D/2$                &$4/2$\\
             &\textup{c}$D/2$                &$2/2$
\end{tabular}
\qquad
\begin{tabular}{@{}clc@{}}
\textup{type}&\multicolumn{1}{c}{$x$}&$a/n$\\
\hline
\textup{e2}  &\textup{c}$D$, \textup{c}$E_6$ &$2$\\
             &\textup{c}$D/2$                &$2/2$\\
\textup{e3}  &\textup{c}$A_2$, \textup{c}$D$ &$3$\\
\multicolumn{3}{c}{}
\end{tabular}
\qquad
\begin{tabular}{@{}clc@{}}
\textup{type}&\multicolumn{1}{c}{$x$}&$a/n$\\
\hline
\textup{e5}  &\textup{c}$E_7$                &$2$\\
\textup{e9}  &\textup{c}$E_{7,8}$            &$2$\\
\textup{e11} &\textup{c}$E/2$                &$2/2$\\
\multicolumn{3}{c}{}
\end{tabular}
\end{table}
\end{enumerate}
\end{theorem}

\section{Termination of divisorial contractions}
The last section is devoted to the proof of Theorem~\ref{thm:sequence}. We keep the notation in the theorem. First of all, we may and shall assume the ground field $k$ to be the field $\bC$ of complex numbers by the same argument as for Theorem~\ref{thm:GKP}. We set $\pi_i\colon E_{i+1}\subset X_{i+1}\to x_i\in X_i$ and $\hat\fm_i=\fm_i\sO_{\hat X_i}$. Let $\sI_i$ denote the ideal sheaf in $\sO_{\hat X_i}$ defining $\hat C_i$. Let $n_i$ denote the index of the germ $x_i\in X_i$ and let $a_i/n_i$ denote the discrepancy $d_{E_{i+1}}(X_i)$ of $\pi_i$, that is, $K_{X_{i+1}}=\pi_i^*K_{X_i}+(a_i/n_i)E_{i+1}$. If $x_i\in X_i$ is not Gorenstein, then we let $e_i$ denote the axial multiplicity of it. We shall define two invariants $m_i$ and $l_i$.

\begin{definition}\label{dfn:mi}
We define $m_i$ to be the least integer such that $\ord_{E_j}\fm_i=m_i$ for infinitely many indices $j>i$. It is well-defined because of the boundedness $\ord_{E_j}\fm_i\le t_i^{-1}$ in terms of the log canonical threshold $t_i$ of $\hat\fm_i$ on $(\hat X_i,\hat\Delta_i)$, as $a_{\hat E_j}(\hat X_i,\hat\Delta_i)=1$.
\end{definition}

\begin{lemma}\label{lem:mi}
$m_{i+1}\le m_i$.
\end{lemma}

\begin{proof}
This follows from the inclusion $\fm_i\sO_{X_{i+1}}\subset\fm_{i+1}$.
\end{proof}

The invariant $l_i$ was originally introduced by Mori \cite{Mo88} as the invariant $i_P(1)$ in the proof of the existence of threefold flips. Let $\Omega_{\hat X_i}'$ denote the sheaf of \textit{special differentials} on $\hat X_i$ introduced in \cite{dFEM11}, by which the canonical divisor $K_{\hat X_i}$ is defined in such a manner that $\sO_{\hat X_i}(K_{\hat X_i})$ coincides with $\bigwedge^3\Omega_{\hat X_i}'$ on the regular locus in $\hat X_i$. Let $\Omega_{\hat C_i}'$ denote that on $\hat C_i$. Then the natural map $\sI_i/\sI_i^{(2)}\times\sI_i/\sI_i^{(2)}\times\Omega_{\hat C_i}'\to\bigwedge^3\Omega_{\hat X_i}'\otimes\sO_{\hat C_i}$ which sends $(x,y,zdu)$ to $zdx\wedge dy\wedge du$, where $\sI_i^{(2)}$ denotes the second symbolic power of $\sI_i$, induces a map
\[
\alpha_i\colon\bigwedge^2\sI_i/\sI_i^{(2)}\to\sHom_{\hat C_i}(\Omega_{\hat C_i}',\gr\omega_i')
\]
of invertible $\sO_{\hat C_i}$-modules, where $\gr\omega_i'$ denotes the quotient of $\sO_{\hat X_i}(K_{\hat X_i})\otimes\sO_{\hat C_i}$ by the maximal torsion submodule.

\begin{definition}
We define $l_i$ to be the length of the cokernel of $\alpha_i$.
\end{definition}

The invariant $l_i$ will be used only if $x_i\in X_i$ is Gorenstein, in which the target of $\alpha_i$ is isomorphic to $\sO_{\hat X_i}(K_{\hat X_i})\otimes\sO_{\hat C_i}(-K_{\hat C_i})$. The germ $\hat x_i\in\hat X_i$ is embedded into the space $\hat A=\Spec k[[z_1,z_2,z_3,z_4]]$ so that $\hat C_i$ is the $z_4$-axis and so that $X_i$ is given by a formal power series of form $z_1z_4^{l_i}+f(z_1,z_2,z_3,z_4)$ with $f\in(z_1,z_2,z_3)^2\sO_{\hat A}$ \cite[lemma~2.16]{Mo88}.

\begin{lemma}\label{lem:li}
Consider the composite $\pi_{ij}\colon x_j\in X_j\to x_i\in X_i$ where $i<j$ such that both $x_i\in X_i$ and $x_j\in X_j$ are Gorenstein. Suppose the existence of generators $z_1,z_2,z_3,z_4$ of $\hat\fm_i$ such that $z_1,z_2,z_3$ generate $\sI_i$ and such that the orders $w_k=\ord_{\hat E_j}z_k$ satisfy $w_1\ge w_2\ge w_3$. If $d_{E_j}(X_i)<w_2+w_3$, then $m_j<m_i$ or $l_j<l_i$.
\end{lemma}

\begin{proof}
We may take $z_4$ generally from $\fm_i$. If $\fm_i\sO_{X_j}$ is contained in $\fm_j^2$, then $m_j<2m_j\le m_i$. We shall assume the contrary $\fm_i\sO_{X_j}\not\subset\fm_j^2$ and prove that $l_j<l_i$. On this assumption, $z_4$ defines locally at $x_j$ a prime divisor. Recall that $X_j$ is factorial at $x_j$ by Lemma~\ref{lem:class}. It follows that $E_j$ is the only $\pi_{ij}$-exceptional prime divisor through $x_j$ and it is given by $z_4$.

Take the commutative diagram
\[
\begin{tikzcd}[baseline=(base.base)]
\bigwedge^2\sI_j/\sI_j^{(2)}\arrow[r,"\alpha_j"]&
\sO_{\hat X_j}(K_{\hat X_j})\otimes\sO_{\hat C_j}(-K_{\hat C_j})\\
\bigwedge^2\sI_i/\sI_i^{(2)}\arrow[r,"\alpha_i"]\arrow[u,"\beta"]&
|[alias=base]|\sO_{\hat X_i}(K_{\hat X_i})\otimes\sO_{\hat C_i}(-K_{\hat C_i})\arrow[u,"\gamma"]
\end{tikzcd}
\]
via $\hat C_j\simeq\hat C_i$, where the vertical maps $\beta$ and $\gamma$ come from the inclusions $\sI_i\sO_{\hat X_j}\subset\sI_j$ and $\pi_{ij}^*\sO_{X_i}(K_{X_i})\subset\sO_{X_j}(K_{X_j})$ respectively. Let $b$ and $g$ be the lengths of the cokernels of $\beta$ and $\gamma$ respectively. Then $b+l_j=l_i+g$. It suffices to show that $g<b$.

We write $a=d_{E_j}(X_i)$ for brevity. The invertible sheaf $\bigwedge^2\sI_i/\sI_i^{(2)}$ is generated by some $z_k\wedge z_l$ with $1\le k<l\le 3$. By assumption, $a<w_k+w_l$. The map $\beta$ sends $z_k\wedge z_l$ to $z_4^{w_k+w_l}z_k'\wedge z_l'$ for the functions $z_k'=z_kz_4^{-w_k}$ and $z_l'=z_lz_4^{-w_l}$ at $\hat x_j$. Then the image of $\beta$ is contained in $z_4^{w_k+w_l}\bigwedge^2\sI_j/\sI_j^{(2)}$ and thus the length $b$ of cokernel is bounded from below by the length $(w_k+w_l)(\hat E_j\cdot\hat C_j)$ of $\sO_{\hat C_j}/z_4^{w_k+w_l}\sO_{\hat C_j}$, where one should recall that $z_4$ defines $E_j$ at $x_j$. On the other hand, it follows from the local equality $\pi_{ij}^*K_{X_i}=K_{X_j}-aE_j$ at $x_j$ that $\gamma$ has cokernel isomorphic to $\sO_{\hat X_j}/\sO_{\hat X_j}(-a\hat E_j)\otimes\sO_{\hat C_j}$ of length $g=a(\hat E_j\cdot\hat C_j)$. Hence $b-g\ge(w_k+w_l-a)(\hat E_j\cdot\hat C_j)>0$.
\end{proof}

We shall investigate the divisorial contraction $\pi_i\colon E_{i+1}\subset X_{i+1}\to x_i\in X_i$. Firstly we treat the case when $x_i\in X_i$ is not Gorenstein.

\begin{lemma}\label{lem:keepge}
Consider the composite $\pi_{ij}\colon x_j\in X_j\to x_i\in X_i$ where $i<j$ such that
\begin{itemize}
\item
$\pi_l$ keeps general elephants for all $i\le l<j$ and
\item
$x_l\in X_l$ is not Gorenstein for all $i<l<j$.
\end{itemize}
Let $S_i$ be the general elephant of $x_i\in X_i$ and let $T_l$ be the strict transform in $X_l$ of $S_i$. Then $(X_l,T_l)$ is crepant to $(X_i,S_i)$ for all $i\le l\le j$.
\end{lemma}

\begin{proof}
By inversion of adjunction \cite[theorem~17.6]{Ko+92}, the pair $(X_i,S_i)$ is at $x_i$ plt and thus canonical as $K_{X_i}+S_i\sim0$. We shall prove the assertion by induction on $l$. Since $\pi_i$ keeps general elephants, $(X_{i+1},T_{i+1})$ is crepant to $(X_i,S_i)$. Fix $i<l<j$ and assume that $(X_l,T_l)$ is crepant to $(X_i,S_i)$. Then $T_l$ is a member of $\abs{-K_{X_l}}$ and passes through the non-Gorenstein point $x_l$. Since $\pi_l$ keeps general elephants, the general elephant $S_l$ of $x_l\in X_l$ satisfies $K_{X_{l+1}}+S_l'=\pi_l^*(K_{X_l}+S_l)$ for the strict transform $S_l'$. Hence $T_l\in\abs{-K_{X_l}}$ satisfies $K_{X_{l+1}}+T_{l+1}\le\pi_l^*(K_{X_l}+T_l)$. This is in fact the equality $K_{X_{l+1}}+T_{l+1}=\pi_l^*(K_{X_l}+T_l)$ because $(X_l,T_l)$ is canonical at $x_l$.
\end{proof}

\begin{lemma}\label{lem:nG}
Suppose that $x_i\in X_i$ is not Gorenstein.
\begin{enumerate}
\item\label{itm:nG-A}
If $x_i\in X_i$ is of type c$A/n_i$, then $x_{i+1}\in X_{i+1}$ is either a quotient singularity, a Gorenstein point or a point of the same type c$A/n_i$ as $x_i\in X_i$ is, and $e_j<e_i$ in the last case.
\item\label{itm:nG-notA}
If $x_i\in X_i$ is not of type c$A/n_i$, then $\pi_i$ keeps general elephants.
\end{enumerate}
\end{lemma}

\begin{proof}
The first assertion is observed from Theorem~\ref{thm:clsA}. In the setting of Theorem~\ref{thm:clsA}(\ref{itm:clsA-ord}), the axial multiplicity of $x\in X$ equals the order of $f(0,x_4)$. As computed on each chart $o_i\in(x_i\neq0)$ in \cite[pp.111--112]{K05}, $Y$ has quotient singularities of types $\frac{1}{r_1}(-1,(a-br_1)/n,1)$ and $\frac{1}{r_2}(-1,(a+br_2)/n,1)$ at $o_1$ and $o_2$ respectively, a singularity of type c$A/n$ at $o_4$ provided that $o_4\in Y$, and possibly Gorenstein singularities in the $x_3x_4$-line. The singularity $o_4\in Y$ is given by $x_1x_2+x_4^{-(r_1+r_2)/n}f(x_3^nx_4^a,x_4)$ in $\fD^4/\bZ_n(1,-1,b,0)$, where the axial multiplicity decreases by $(r_1+r_2)/n$.

The second assertion follows from Theorem~\ref{thm:keepge}.
\end{proof}

\begin{corollary}\label{crl:nG}
For infinitely many indices $i$, $x_i\in X_i$ is either a quotient singularity or Gorenstein.
\end{corollary}

\begin{proof}
Assuming that $x_i\in X_i$ is not Gorenstein, we shall find an index $j>i$ such that $x_j\in X_j$ is a quotient singularity or Gorenstein.

Fix an index $l>i$. If $x_k\in X_k$ is a non-Gorenstein point of type other than c$A/n_k$ for all $i\le k\le l$, in which $n_k=2$, $3$ or $4$, then by Lemma~\ref{lem:nG}(\ref{itm:nG-notA}), $\pi_k$ keeps general elephants for all $i\le k<l$. Let $S_i$ be the general elephant of $x_i\in X_i$ and let $T_k$ be the strict transform in $X_k$ of it. By Lemma~\ref{lem:keepge}, $(X_k,T_k)$ is crepant to $(X_i,S_i)$ and in particular $T_k\in\abs{-K_{X_k}}$ passes through the non-Gorenstein point $x_k$. Similarly to Definition~\ref{dfn:mi}, we consider the least rational number $s_k\in12^{-1}\bZ$ such that $\ord_{E_{k'}}T_k=s_k$ for infinitely many indices $k'>k$. It follows from the equality $\pi_k^*T_k=T_{k+1}+(a_k/n_k)E_{k+1}$ that $s_{k+1}<s_k$ and thus $s_{k+1}\le s_k-1/12$. Hence $l<i+12s_i$ and one attains an index $i'\ge i$ such that $x_{i'}\in X_{i'}$ is a non-Gorenstein point of type c$A/n_{i'}$ or Gorenstein.

Now assume that $x_{i'}\in X_{i'}$ is a non-Gorenstein point of type c$A/n_{i'}$ but not a quotient singularity. Fix an index $l>i'$. If $x_k\in X_k$ is neither a quotient singularity nor Gorenstein for all $i'\le k\le l$, then by Lemma~\ref{lem:nG}(\ref{itm:nG-A}), one inductively observes that $x_{k+1}\in X_{k+1}$ is of type c$A/n_{i'}$ with $e_{k+1}<e_k$ for all $i'\le k<l$. Hence $l<i+e_{i'}$ and together with the same lemma one attains an index $j\ge i'$ such that $x_j\in X_j$ is a quotient singularity or Gorenstein.
\end{proof}

Secondly we treat the case when $x_i\in X_i$ is Gorenstein.

\begin{lemma}\label{lem:ordI}
Suppose that both $x_i\in X_i$ and $x_{i+1}\in X_{i+1}$ are Gorenstein and that $m_{i+1}=m_i$. If there exists a pair $(w_2,w_3)$ of positive integers with $w_2\ge w_3$ such that the dimension $d(l)$ of the vector space $V_l=\pi_{i*}\sO_{X_{i+1}}(-lE_{i+1})/\pi_{i*}\sO_{X_{i+1}}(-(l+1)E_{i+1})$ is $\rd{l/w_3}+1$ for all $0\le l<w_2$, then $\ord_{\hat E_{i+1}}\sI_i\ge w_3$. If moreover $w_2>w_3$, then $\ord_{\hat E_{i+1}}\sI_i=w_3$ and $\sI_i$ has generators $z_1,z_2,z_3$ such that $\ord_{\hat E_{i+1}}z_k\ge w_2$ for $k=1,2$.
\end{lemma}

\begin{proof}
We identify $V_l$ with $V_l\otimes_{\sO_{X_i}/\fm_i}\sO_{\hat X_i}/\hat\fm_i$. We write $E=E_{i+1}$ and $\hat E=\hat E_{i+1}$ for brevity. The assertion is obvious if $w_2=1$. We assume that $w_2\ge2$. Following the argument for \cite[lemma~3.5.2]{K24}, one finds generators $z_1,z_2,z_3,z_4$ of $\fm_i$ such that $\ord_Ez_4=1$ and $\ord_Ez_k\ge w_3$ for $k=1,2,3$. Indeed, since $d(1)\ge1$, a general element $z_4$ in $\fm_i$ has $\ord_Ez_4=1$. The monomial $z_4^l$ forms a basis of $V_l\simeq k$ as far as $l<w_3$. It follows that for $k=1,2,3$, $z_k$ is congruent modulo $\pi_{i*}\sO_{X_{i+1}}(-w_3E)$ to some polynomial $p_k(z_4)$ in $z_4$. Replace $z_k$ by $z_k-p_k$.

Further if $w_2>w_3$, then one finds $z_1,z_2,z_3,z_4$ so that $\ord_Ez_3=w_3$ and $\ord_Ez_k\ge w_2$ for $k=1,2$. Indeed, since $d(w_3)=2$, after permutation of $z_1,z_2,z_3$ the monomials $z_3$ and $z_4^{w_3}$ form a basis of $V_{w_3}\simeq k^2$, in which $\ord_Ez_3=w_3$. The order along $E$ of every weighted homogeneous polynomial in $z_3,z_4$ coincides with the weighted degree, because it is factorised as the product of polynomials of form either $z_3+\lambda z_4^{w_3}$ or $z_4$. It follows that the monomials $z_3^{c_3}z_4^{c_4}$ with $w_3c_3+c_4=l$ are linearly independent in $V_l$. As far as $l<w_2$, $d(l)$ equals the number of these monomials and hence they form a basis of $V_l$. Thus for $k=1,2$, $z_k$ is congruent modulo $\pi_{i*}\sO_{X_{i+1}}(-w_2E)$ to some polynomial $q_k(z_3,z_4)$ in $z_3,z_4$. Replace $z_k$ by $z_k-q_k$ so that $\ord_Ez_k\ge w_2$.

Recall that $z_4$ is taken generally from $\fm_i$. It follows from the assumption $m_{i+1}=m_i$ that $\fm_i\sO_{X_{i+1}}$ is not contained in $\fm_{i+1}^2$. For the same reason as at the beginning of the proof of Lemma~\ref{lem:li}, $E$ is locally at $x_{i+1}$ given by $z_4$.

We shall prove that $\ord_{\hat E}\sI_i\ge w_3$. Otherwise, there would exist an element $f$ in $\sI_i$ whose order $w$ along $\hat E$ is less than $w_3$. Since the divisor $\hat F$ on $\hat X_i$ defined by $f$ contains $\hat C_i$, its strict transform $\hat F'$ contains $\hat C_{i+1}$ and thus the function $fz_4^{-w}$ defining $\hat F'$ at $\hat x_{i+1}$ vanishes there, that is, $fz_4^{-w}\in\hat\fm_{i+1}$. On the other hand, since $V_w=kz_4^w$, there exists a non-zero constant $\lambda\in k$ such that $h=f-\lambda z_4^w$ is of order greater than $w$ along $\hat E$ and in particular $hz_4^{-w}\in\hat\fm_{i+1}$. Then $\lambda=fz_4^{-w}-hz_4^{-w}\in\hat\fm_{i+1}$, which is absurd.

Suppose that $w_2>w_3$. Since $\hat\fm_i=\sI_i+z_4\sO_{\hat X_i}$, the vector space $V_{w_3}\simeq k^2$ is spanned by $\sI_i$ and $z_4^{w_3}$. Thus $\ord_{\hat E}\sI_i=w_3$ and one finds generators $\hat z_1,\hat z_2,\hat z_3,z_4$ of $\hat\fm_i$ such that $\hat z_1,\hat z_2,\hat z_3$ generate $\sI_i$ and such that $\hat z_3$ and $z_4^{w_3}$ form a basis of $V_{w_3}$. Further for $k=1,2$, there exists a polynomial $\hat q_k(\hat z_3,z_4)$ in $\hat z_3,z_4$ such that $\hat z_k-\hat q_k$ is of order at least $w_2$ along $\hat E$. No monomial in $z_4$ appears in $\hat q_k$ for the same reason as in the preceding paragraph. Indeed, if one had an expression $\hat q_k=\hat z_3r(\hat z_3,z_4)+s(z_4)$ such that the order $w$ along $\hat E$ of $s$ is less than $w_2$, then $\hat z_k-\hat z_3r=s+(\hat z_k-\hat q_k)\in\sI_i$ would also be of order $w$ along $\hat E$ and $sz_4^{-w}=(\hat z_k-\hat z_3r)z_4^{-w}-(\hat z_k-\hat q_k)z_4^{-w}$ would belong to $\hat\fm_{i+1}$, which is absurd. Hence $\hat q_k\in\hat z_3\sO_{\hat X_i}\subset\sI_i$. Replacing $\hat z_k$ by $\hat z_k-\hat q_k$, one attains the inequality $\ord_{\hat E}z_k\ge w_2$.
\end{proof}

\begin{proposition}\label{prp:G-G}
If both $x_i\in X_i$ and $x_{i+1}\in X_{i+1}$ are Gorenstein and $x_i\in X_i$ is singular, then $m_{i+1}<m_i$ or $l_{i+1}<l_i$.
\end{proposition}

\begin{proof}
If $a_i=1$, then the assertion follows from Lemma~\ref{lem:li}. Assuming that $a_i\ge2$ and that $m_{i+1}=m_i$, we shall prove that $l_{i+1}<l_i$.

If $\pi_i$ is of ordinary type as in Table~\ref{tbl:numerical}, then by Lemma~\ref{lem:dim}(\ref{itm:dim-ord}), the pair $(w_2,w_3)=(\min\{a_i,r_2\},\min\{a_i,r_1\})$ satisfies the assumption in Lemma~\ref{lem:ordI} with $a_i<w_2+w_3$, where we assume that $r_1\le r_2$ and where we set $(r_1,r_2)=(1,r)$ if type o2 and set $(r_1,r_2)=(1,1)$ if type o1. Thus the inequality $l_{i+1}<l_i$ follows from Lemma~\ref{lem:li}.

If $\pi_i$ is of exceptional type, then $d(1)\le1$ and $a_i\le4$ from Lemma~\ref{lem:dim}(\ref{itm:dim-exc}) and Table~\ref{tbl:notmin} respectively. If $d(1)=0$, then $\fm_i\sO_{X_{i+1}}\subset\sO_{X_{i+1}}(-2E_{i+1})\subset\fm_{i+1}^2$, which contradicts our assumption $m_{i+1}=m_i$. Hence $d(1)=1$. If $a_i=2$ or $3$, then we set $(w_2,w_3)=(2,2)$. If $a_i=4$, which occurs only if $\pi_i$ is of type e1 in Table~\ref{tbl:notmin}, then we set $(w_2,w_3)=(3,3)$ or $(3,2)$ according to $d(2)=1$ or $2$ in Lemma~\ref{lem:dim}(\ref{itm:dim-e1}). The pair $(w_2,w_3)$ satisfies the assumption in Lemma~\ref{lem:ordI} with $a_i<w_2+w_3$. Hence the inequality $l_{i+1}<l_i$ follows from Lemma~\ref{lem:li}.
\end{proof}

\begin{lemma}\label{lem:G-nG}
If $x_i\in X_i$ is Gorenstein but $x_{i+1}\in X_{i+1}$ is not Gorenstein, then either $x_{i+1}\in X_{i+1}$ is a quotient singularity, $x_{i+2}\in X_{i+2}$ is at most a quotient singularity or $\pi_i$ keeps general elephants.
\end{lemma}

\begin{proof}
If $a_i=1$, then $\pi_i$ keeps general elephants. Assuming that $a_i\ge2$ and that $x_{i+1}\in X_{i+1}$ is not a quotient singularity, we shall prove that $x_{i+2}\in X_{i+2}$ is at most a quotient singularity.

Since $x_i\in X_i$ is Gorenstein, the non-Gorenstein point $x_{i+1}\in X_{i+1}$ is not hidden in the sense of Definition~\ref{dfn:hidden}. Thus $x_{i+1}\in X_{i+1}$ contributes to the set $J$ defined from $\pi_i$ as in Section~\ref{sct:divcont}. The list of $J$ is in Table~\ref{tbl:numerical}. Since $x_{i+1}\in X_{i+1}$ is not a quotient singularity, $\pi_i$ must be of exceptional type. It follows from the relation $K_{X_{i+1}}\sim a_iE_{i+1}$ at $x_{i+1}$ that $n_{i+1}$ is coprime to $a_i$. Since $a_i\ge2$, from Table~\ref{tbl:notmin} for $(\pi_i,a_i)$ and from Table~\ref{tbl:numerical} for $x_{i+1}$, the triple $(\pi_i,a_i,x_{i+1})$ is $(\textrm{e2},2,\textrm{c$A/n_{i+1}$})$ with $n_{i+1}$ odd, $(\textrm{e2},2,\textrm{c$D/3$})$ or $(\textrm{e3},3,\textrm{c$Ax/4$})$. By Lemma~\ref{lem:dim}(\ref{itm:dim-exc}), $\ord_{E_{i+1}}\fm_i=1$ in every case.

In the cases $(\textrm{e2},2,\textrm{c$D/3$})$ and $(\textrm{e3},3,\textrm{c$Ax/4$})$, where $n_{i+1}=a_i+1$, it follows from Theorem~\ref{thm:notmin} that $\pi_{i+1}$ has discrepancy $1/(a_i+1)$. Since $K_{X_{i+2}}=(\pi_i\circ\pi_{i+1})^*K_{X_i}+a_i\pi_{i+1}^*E_{i+1}+1/(a_i+1)E_{i+2}$, one has
\[
d_{E_{i+2}}(X_i)=a_i\ord_{E_{i+2}}E_{i+1}+\frac{1}{a_i+1}\in\frac{a_i\bZ+1}{a_i+1}\cap\bZ=1+a_i\bZ
\]
and can write $\ord_{E_{i+2}}E_{i+1}=c+1/(a_i+1)$ and $d_{E_{i+2}}(X_i)=a_ic+1$ by some non-negative integer $c$. Then $\ord_{E_{i+2}}\fm_i\ge\ru{\ord_{E_{i+2}}E_{i+1}}=c+1$ and
\[
\frac{d_{E_{i+2}}(X_i)}{\ord_{E_{i+2}}\fm_i}\le\frac{a_ic+1}{c+1}<a_i=\frac{d_{E_{i+1}}(X_i)}{\ord_{E_{i+1}}\fm_i},
\]
which contradicts the last condition in Proposition~\ref{prp:sequence}.

In the remaining case $(\textrm{e2},2,\textrm{c$A/n_{i+1}$})$, it follows from Theorem~\ref{thm:clsA} that $x_{i+1}\in X_{i+1}$ is analytically given by a function $z_1z_2+f(z_3^{n_{i+1}},z_4)$ in $\fD^4/\bZ_{n_{i+1}}(1,-1,b,0)$ and $\pi_{i+1}$ is the weighted blow-up with $\wt(z_1,z_2,z_3,z_4)=\frac{1}{n_{i+1}}(r_1,r_2,a_{i+1},n_{i+1})$. Since $\pi_i$ has $J=\{(n_{i+1},1),(n_{i+1},1)\}$, the axial multiplicity $e_{i+1}$ of $x_{i+1}\in X_{i+1}$ is two. Thus $z_4^2$ appears in $f$ and the weighted order $(r_1+r_2)/n_{i+1}$ of $f$ is one or two. Since $a_{i+1}n_{i+1}$ divides $r_1+r_2=n_{i+1}$ or $2n_{i+1}$, one has $a_{i+1}=1$ or $2$.

If $a_{i+1}=1$, then by the same argument as above, $d_{E_{i+2}}(X_i)=2\ord_{E_{i+2}}E_{i+1}+1/n_{i+1}$ and one can write $\ord_{E_{i+2}}E_{i+1}=c+(n_{i+1}-1)/2n_{i+1}$ and $d_{E_{i+2}}(X_i)=2c+1$. Then $d_{E_{i+2}}(X_i)/\ord_{E_{i+2}}\fm_i<2=d_{E_{i+1}}(X_i)/\ord_{E_{i+1}}\fm_i$, which is a contradiction. Hence $a_{i+1}=2$, $r_1+r_2=2n_{i+1}$ and $z_3^{n_{i+1}}$ appears in $f$. Then replacing coordinates, one attains the expression $f=z_3^{n_{i+1}}+z_4^2$. Now by a direct computation of the weighted blow-up $\pi_{i+1}$, one can check that $X_{i+2}$ has only quotient singularities.
\end{proof}

\begin{proposition}\label{prp:G-nG-G}
Consider the composite $\pi_{ij}\colon x_j\in X_j\to x_i\in X_i$ to a Gorenstein point $x_i\in X_i$ where $i<j$ such that
\begin{itemize}
\item
$\pi_l$ keeps general elephants for all $i\le l<j$,
\item
$x_l\in X_l$ is not Gorenstein for all $i<l<j$ and
\item
either $\pi_j$ does not keep general elephants or $x_j\in X_j$ is Gorenstein.
\end{itemize}
Then one of the following holds.
\begin{enumerate}
\item
$m_j<m_i$.
\item
$x_j\in X_j$ is Gorenstein and $l_j<l_i$.
\item
$x_j\in X_j$ or $x_{j+1}\in X_{j+1}$ is at most a quotient singularity.
\end{enumerate}
\end{proposition}

\begin{proof}
The general elephant $S_i$ of the Gorenstein point $x_i\in X_i$ is nothing but the general hyperplane section through $x_i$. Let $T_j$ be the strict transform in $X_j$ of $S_i$. By Lemma~\ref{lem:keepge}, $K_{X_j}+T_j=\pi_{ij}^*(K_{X_i}+S_i)$. Define the $\pi_{ij}$-exceptional Weil divisor $D$ by the equality $K_{X_j}=\pi_{ij}^*K_{X_i}+D$. Then $\pi_{ij}^*S_i=T_j+D$ and hence $\fm_i\sO_{X_j}\subset\sO_{X_j}(-D)$. Note that $E_j\le D$.

Suppose that $x_j\in X_j$ is Gorenstein. If $D$ is prime at $x_j$, then $d_{E_j}(X_i)=1$ and Lemma~\ref{lem:li} yields $m_j<m_i$ or $l_j<l_i$. If $D$ is not prime at $x_j$, then $\fm_i\sO_{X_j}\subset\sO_{X_j}(-D)\subset\fm_j^2$ and thus $m_j<2m_j\le m_i$.

Suppose that $x_j\in X_j$ is not Gorenstein, in which $x_j\in T_j$. If it is hidden, that is, $E_j$ is Cartier at $x_j$, then the ideal $\sO_{X_j}(-\pi_{ij}^*S_i)$ in $\sO_{X_j}$ is at $x_j$ decomposed as the product of non-trivial ideals $\sO_{X_j}(-E_j)$ and $\sO_{X_j}(-T_j-(D-E_j))$. This implies that $\fm_i\sO_{X_j}\subset\fm_j^2$ and $m_j<m_i$. Henceforth we shall assume that $x_j\in X_j$ is not hidden.

By assumption, $\pi_j$ does not keep general elephants and in particular $a_j/n_j>1$. By Lemma~\ref{lem:nG}(\ref{itm:nG-notA}), $x_j\in X_j$ is of type c$A/n_j$. We may assume that it is not a quotient singularity. Since $x_j\in X_j$ contributes to the set $J$ defined from $\pi_{j-1}$ as a subset of form $e_j\times(n_j,v)$ where $e_j$ is the axial multiplicity, one sees from Table~\ref{tbl:numerical} that the triple $(\pi_{j-1},n_j,e_j)$ is $(\textrm{e2},n_j,2)$, $(\textrm{e13},2,3)$, $(\textrm{e13},2,2)$ or $(\textrm{e14},3,2)$.

The following contraction $\pi_j$ is to the point $x_j\in X_j$ of type c$A/n_j$ and it is completely classified as in Theorem~\ref{thm:clsA}. It follows that $a_j\le e_j$. Indeed, for the analytic function $z_1z_2+f(z_3^{n_j},z_4)$ defining $x_j\in X_j$, the weighted order of $f$ is a multiple of $a_j$ and at most $e_j$. Since $a_j/n_j>1$, the triple must be $(\textrm{e13},2,3)$ with $a_j/n_j=3/2$. In this case, $x_j\in X_j$ is analytically given by $z_1z_2+f(z_3^2,z_4)$ in $\fD^4/\bZ_2(1,1,1,0)$ and $\pi_j$ is the weighted blow-up with $\wt(z_1,z_2,z_3,z_4)=\frac{1}{2}(5,1,3,2)$, where $f$ is of weighted order three and $z_3^2$ and $z_4^3$ appear in $f$. Replacing coordinates, one attains the expression $f=z_3^2+z_4^3$ and can directly check that $X_{j+1}$ has only quotient singularities.
\end{proof}

\begin{proof}[Proof of Theorem~\textup{\ref{thm:sequence}}]
Note that if $x_i\in X_i$ is not smooth but a quotient singularity, then as shown by Kawamata \cite{Ka96}, the following contraction $\pi_i$ is uniquely determined as a weighted blow-up and $x_{i+1}\in X_{i+1}$ is at most a quotient singularity of index less than $n_i$. Hence $x_j\in X_j$ is smooth for some index $i<j<i+n_i$.

We shall derive a contradiction by assuming that $x_i\in X_i$ is neither smooth nor a quotient singularity for all $i$. By Lemma~\ref{lem:mi}, there exists an index $i_0$ such that $m_i$ is constant for $i\ge i_0$. By Corollary~\ref{crl:nG}, it suffices to show that for $i\ge i_0$, if $x_i\in X_i$ is Gorenstein, then there exists an index $j>i$ such that $x_j\in X_j$ is Gorenstein and such that $l_j<l_i$. If $x_{i+1}\in X_{i+1}$ is Gorenstein, then the inequality $l_{i+1}<l_i$ follows from Proposition~\ref{prp:G-G}. If $x_{i+1}\in X_{i+1}$ is not Gorenstein, then it follows from Lemma~\ref{lem:G-nG} together with Corollary~\ref{crl:nG} that there exists a composite $\pi_{ij}\colon x_j\in X_j\to x_i\in X_i$ for some $j>i$ which satisfies the assumptions in Proposition~\ref{prp:G-nG-G}. Then $x_j\in X_j$ is Gorenstein and $l_j<l_i$.
\end{proof}

We remark that our induction does not apply to the study of minimal log discrepancies on a smooth threefold because the invariant $l_i$ is zero whenever $x_i\in X_i$ is smooth.

\section*{Acknowledgements}
This work was inspired during my visit to Peking University, where I discussed with Jihao Liu many remaining problems on threefold singularities. I wish to express my gratitude to him for the valuable discussions as well as his warm hospitality.

\end{document}